\documentclass[oneside]{amsart}
\usepackage{amsthm,amsmath,amssymb,mathrsfs}
\usepackage[usenames,dvipsnames]{xcolor}
\usepackage{graphicx}
\usepackage{calrsfs}
\usepackage{hyperref}
\usepackage{enumerate}
\usepackage[T1]{fontenc}
\usepackage{stackengine,scalerel}
\usepackage{lmodern}
\usepackage[foot]{amsaddr}
\usepackage{csquotes}

\usepackage{xcolor,soul}

\usepackage[a4paper]{geometry}

\usepackage{cite}

\numberwithin{equation}{section}

\theoremstyle{plain}
\newtheorem{theorem}{Theorem}[section]

\newtheorem{corollary}[theorem]{Corollary}
\newtheorem{lemma}[theorem]{Lemma}

\theoremstyle{definition}

\newtheorem{remark}[theorem]{Remark}
\newtheorem{example}[theorem]{Example}
\newtheorem{examples}[theorem]{Examples}

\DeclareMathOperator*{\esssup}{ess\,sup}

\DeclareMathOperator{\supp}{supp}
\DeclareMathOperator{\diam}{diam}

\DeclareMathOperator{\loc}{loc}

\usepackage{esint}


\usepackage{tikz}
\usepackage{atbegshi}

\AtBeginShipoutFirst{%
	\begin{tikzpicture}[remember picture, overlay]
		\node[anchor=north west, xshift=1in, yshift=-1.5cm] at (current page.north west) {%
			\includegraphics[width=2cm]{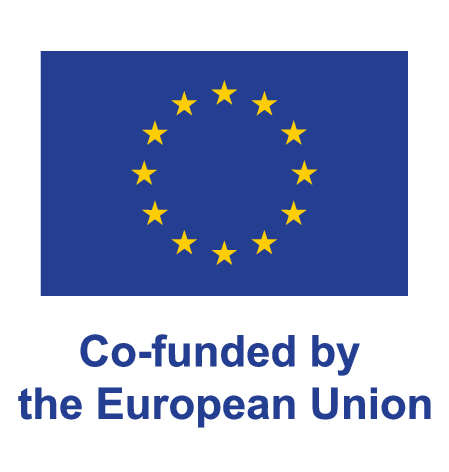} 
		};
	\end{tikzpicture}%
}

\begin{document}
	
	\newcommand\Humlaut[1]{\stackengine{-.05ex}{#1}{\hstretch{.8}{\vstretch{.65}{%
					\mkern1mu\scriptscriptstyle''}}}{O}{c}{F}{F}{S}}
	
	\author[M. Hinz]{Michael Hinz}
	\email{mhinz@math.uni-bielefeld.de}
	\author[J. M. T\"olle]{Jonas M. T\"olle}
	\email{jonas.tolle@aalto.fi}
	\author[L. Viitasaari]{Lauri Viitasaari}
	\email{lauri.viitasaari@aalto.fi}
	
	\address[MH]{Bielefeld University, Faculty of Mathematics, Postfach 100131, 33501 Bielefeld, Germany}
	\address[JMT]{Aalto University, Department of Mathematics and Systems Analysis, P.O. Box 11100 (Otakaari 1, Espoo), 00076 Aalto, Finland}
	\address[LV]{Aalto University, Department of Information and Service Management, P.O. Box 21210 (Ekonominaukio 1, Espoo), 00076 Aalto, Finland}

	\title{On fractal minimizers and potentials of occupation measures}

	\begin{abstract}
		We highlight two basic types of variational problems and prove the existence of fractal minimizers through the direct method in the calculus of variations. By design these minimizers are H\"older curves or H\"older parametrizations of surfaces whose images generally have a non-integer Hausdorff dimension. Their origin is deterministic, but their regularity properties are roughly similar to those of typical realizations of stochastic processes. As a key tool and our main technical result, we prove novel continuity and boundedness results for potentials of occupation measures of Gaussian random fields. These results complement well-known results for local times, but hold under much less restrictive assumptions. In an auxiliary section, we generalize earlier results on nonlinear compositions of fractional Sobolev functions with $BV$-functions to higher dimensions.
	\end{abstract}
	
	\maketitle

	{\small
		\tableofcontents
	}

	\section{Introduction}

	The direct method in the calculus of variations \cite{Dacorogna2004, JostLiJost} ensures the existence of minimizers for an energy functional over a given class of objects, and it merely needs some compactness of the class of admissible objects and the lower semicontinuity of the 
	functional to be minimized. In many cases a minimizer is a curve or hypersurface with a smooth parametrization. Also equilibrium problems in potential theory \cite{BHS19,Helms,Landkof} may be viewed as examples of the direct method: A charge, constrained to a set, is represented by a measure on that set and to minimize an energy functional, the measure has to be sufficiently diffuse.

	We are interested in parametrizations whose images are fractal, i.e., non-smooth in a very strong sense. Although far from smooth, they may have some strong geometric homogeneity features. A \enquote{perfect} example is the classical self-similar Koch curve \cite{Falconer}, whose natural parametrization is H\"older continuous of order $\gamma=\frac{\log 3}{\log 4}$ and each arc of which has positive and finite $\frac{1}{\gamma}$-dimensional Hausdorff measure. It seems natural to ask whether H\"older curves or surfaces in $\mathbb{R}^n$ having a non-integer Hausdorff dimension within a certain predetermined range can occur as minimizers within the context of the direct method. If so, this would mean that they arise from a purely deterministic (and non-iterative) mechanism.
	
	We consider two basic types of variational problems and show the existence of minimizers which, for suitable parameters, by design must be fractal --- rectifiable curves or surfaces would violate given constraints or produce infinite functional values. The first is the minimization of a repulsive self-interaction energy \eqref{E:interaction} over a class of H\"older parametrizations $X:[0,1]^k\to \mathbb{R}^n$ with $X(0)=0$ and bounded H\"older constant, see Theorem \ref{T:minselfinteraction}. For the second problem a medium in the form of a measure $\nu$ on $\mathbb{R}^n$ is given, and we look for parametrizations $X$ which move \enquote{neither too slow nor too fast} and minimize a mutual interaction energy \eqref{E:repattenergy} with respect to $\nu$, see Theorem \ref{T:mininteraction}. This can be interpreted as traveling along a moderately homogeneous trajectory within or close to the medium $\nu$ and with a minimal interaction. 
	
	The mentioned problems involve classes of parametrizations $X$ having a quantified minimal H\"older regularity $0<\gamma<1$ and, simultaneously, a quantified minimal irregularity, which can be described using the regularity of their occupation measures $\mu_X$ \cite{Berman69a, GemanHorowitz, GG20b, HTV2021-1}, see \eqref{E:occumeas} below. A sufficient condition for the regularity of $\mu_X$ is the finiteness of its $\alpha$-Riesz energy (or the boundedness of its $\alpha$-Riesz potential). We point out a simple and robust compactness result, tailor-made for such situations: An application of the Arzel\`a--Ascoli theorem to a minimizing sequence of parametrizations gives a uniformly convergent subsequence; the H\"older \emph{regularity is preserved} in the limit. Since the corresponding occupation measures converge weakly, the lower semicontinuity of Riesz energies (or potentials) implies that also the \emph{irregularity} of the parametrizations \emph{is preserved} in the limit. See Lemma \ref{L:compact} and the proof of Theorem \ref{T:mininteraction}.
	
	An essential part of the existence proof is the observation that the classes of parametrizations $X$ under consideration are nonempty. For large enough ambient space dimensions $n$ this is immediate from Assouad's embedding theorem \cite{Assouad83, Heinonen01}, which states that there is a function $X:[0,1]^k\to\mathbb{R}^n$ and there are constants $C>c>0$ such that 
	\begin{equation}\label{E:Assouad}
		c|t-s|^\gamma\leq |X(t)-X(s)|\leq C|t-s|^\gamma,\qquad s,t\in [0,1]^k,
	\end{equation}
	see \cite[Proposition 4.4]{Assouad83}. The function $X$ in \cite[Proposition 4.4]{Assouad83} is a (multivariate) Koch-type curve with $\dim_H X([0,1]^k)=\frac{k}{\gamma}$, a \enquote{perfect} example of the type of parametrization we are looking for. However, Assouad's theorem comes with strong restrictions on the parameter ranges. To widen the parameter ranges and to show that indeed such parametrizations occur most naturally, we provide \emph{novel results for potentials of occupation measures}. 
	
	The study of potentials of occupation measures is our main technical contribution in this article. There is a considerable amount of literature on the local times of deterministic functions or stochastic processes, see for instance \cite{Berman69, Berman69a, Berman70, Berman73, GemanHorowitz, Pitt78} for classical results and \cite{Xiao97,Xiao06, Ayache2011} for more recent ones. Occupation measures themselves seem to have received less attention; some classical results can be found in \cite{Adler81, CiesielskiTaylor62, Demboetal00, Kahane, PerkinsTaylor87, Ray64}, and some contemporary discussions and applications in \cite{CatellierGubinelli,GG20b, HarangPerkowski}. Results on the existence, continuity or boundedness of local times typically come with severe dimensional restrictions. Brownian motion in $\mathbb{R}^n$, for example, has local times if and only if $n=1$; in this case these local times are a.s. continuous. Partly following our former work in \cite{HTV2020, HTV2021-1, HTV2025}, we propose to study potentials of occupation measures and to regard them, where appropriate, as substitutes for local times. In particular for the study of path properties this works very well; some related new results are discussed in \cite{HTV2026+}. Here we show that classical results for local times by Berman \cite{Berman69} and Pitt \cite{Pitt78} have natural counterparts for potentials of occupation measures, see  Corollary \ref{C:Fourier} (i) and Theorem \ref{T:Pitt}. Applications to fractional Brownian fields are discussed in Lemmas \ref{L:Bermanbasic} and \ref{L:BermanPitt}. Compared to the local time case, dimensional restrictions for the continuity of potentials of occupation measures are much weaker. For instance, for $n\geq 3$ and any $n-2<\alpha<n$ the $\alpha$-Riesz potential of the occupation measure of Brownian motion in $\mathbb{R}^n$ is a.s. continuous; this can be concluded using \cite{Demboetal00, PerkinsTaylor87}; its boundedness also follows from Lemma \ref{L:BermanPitt} (i) for $H=\frac12$ and $k=1$.

	In an auxiliary section, independent of the aforementioned results, we provide a generalization of a multiplicative estimate for nonlinear compositions $\varphi\circ u$ of $\mathbb{R}^n$-valued fractional Sobolev functions $u$ with $BV$-functions $\varphi$, see Theorem \ref{T:main}.
	In \cite{HTV2020, HTV2021-1} we proved this result for curves $u$; here we allow functions $u:G\to\mathbb{R}^n$ defined on general bounded domains $G$ in $\mathbb{R}^k$. Theorem \ref{T:main} could be used to study the existence of solutions to stochastic partial differential equations with irregular coefficients, in the spirit of \cite{HTV2020}. The main sufficient condition for $\varphi\circ u$ to be well defined and to satisfy the multiplicative estimate is the finiteness of a nonlinear energy of the gradient measure $D\varphi$ of $\varphi$ and the occupation measure $\mu_u$ of $u$; it ensures that only a small part of the image of $u$ is located where $D\varphi$ is very concentrated. 
	
	In Section \ref{S:minselfinteract}, we show the existence of H\"older parametrizations minimizing a Riesz energy. In Section \ref{S:velocity}, we discuss \enquote{velocity bands} for fractal curves and prove an existence result for H\"older parametrizations minimizing a mutual Riesz energy. Continuity results for potentials of occupation measures are proved in Section \ref{S:occupot}. Section \ref{S:comp} contains the composition result. 
	
	We briefly fix some notation and preliminaries. For $x\in\mathbb{R}^n$ and $r>0$ we let $B(x,r)$ denote the open ball in $\mathbb{R}^n$ with center $x$ and radius $r$; we use the notation $\overline{B}(x,r)$ for the closed ball. Given $s\geq 0$, we write $\mathcal{H}^s$ for the $s$-dimensional Hausdorff measure on $\mathbb{R}^n$ and $\dim_H$ for the Hausdorff dimension, see \cite{Falconer}.
	
	Let $n\in\mathbb{N}\setminus\{0\}$. Given $0<\alpha<n$, the \emph{Riesz kernel of order $\alpha$} is $k_\alpha(x):=|x|^{\alpha-n}$, $x\in \mathbb{R}^n\setminus\{0\}$;
	for convenience we omit the customary multiplicative constant. Given a nonnegative Radon measure $\mu$ on $\mathbb{R}^n$, we consider the \emph{Riesz potential of order $\alpha$}, defined by
	\begin{equation}\label{E:Rieszpot}
		U^\alpha\mu(x):= \int_{\mathbb{R}^n}k_{\alpha}(x-y)\mu(dy),\quad x\in\mathbb{R}^n. 
	\end{equation}
	Given two nonnegative Radon measures $\mu$ and $\nu$ on $\mathbb{R}^n$, their \emph{mutual Riesz energy of order $\alpha$} is
	\begin{equation}\label{E:repattenergy}
		I^\alpha(\mu,\nu):=\int_{\mathbb{R}^n}U^\alpha\mu(x)\nu(dx)=\int_{\mathbb{R}^n}\int_{\mathbb{R}^n}k_{\alpha}(x-y)\mu(dy)\nu(dx).
	\end{equation}
	See \cite{AH96} or \cite{Landkof}. The mutual interaction energy $I^\alpha(\mu,\nu)$ may be interpreted as the potential energy caused by a soft singular repulsion between $\mu$ and the given medium $\nu$ at small scales. In the case $\nu=\mu$ the \emph{Riesz energy of order $\alpha$ of $\mu$}, 
	\begin{equation}\label{E:interaction}
		I^\alpha(\mu):=I^\alpha(\mu,\mu),
	\end{equation}
	is the self-interaction energy of $\mu$. In this case $I^\alpha(\mu)$ describes a soft singular self-repulsion. See \cite[Section 2.1]{BHS19} and \cite{Ch13} for related topics.

	Let $k\in \mathbb{N}\setminus\{0\}$ and suppose that
	\begin{equation}\label{E:function}
		X:[0,1]^k\to\mathbb{R}^n
	\end{equation}
	is a Borel function. Its \emph{occupation measure} \cite{Berman69, Berman69a, GemanHorowitz} is the Borel probability measure on $\mathbb{R}^n$ defined by
	\begin{equation}\label{E:occumeas}
		\mu_X(B):=\mathcal{L}^k(\{t\in [0,1]^k: X(t)\in B\}),\quad \text{$B\subset \mathbb{R}^n$ Borel.}
	\end{equation}
	Note that the map $X\mapsto \mu_X$ is highly nonlinear. If $\mu_X$ is absolutely continuous with respect to $\mathcal{L}^n$, then its density 
	$L_X:=\frac{d\mu_X}{d\mathcal{L}^n}\in L^1(\mathbb{R}^n)$ 
	is called the \emph{local time} of $X$. We will discuss only the case of continuous functions $X$; in this case
	\begin{equation}\label{E:contcasesupport}
		\supp \mu_X=X([0,1]^k).
	\end{equation}
	
	Given $\gamma\in (0,1)$, we write $\mathcal{C}_0^\gamma([0,1]^k,\mathbb{R}^n)$ for the Banach space of functions $X$ as in \eqref{E:function}  satisfying $X(0)=0$ and
	\begin{equation}\label{E:Hoelder}
		\|X\|_{\mathcal{C}_0^\gamma}:=\sup_{s,t\in [0,1]^k,\ s\neq t}\frac{|X(t)-X(s)|}{|t-s|^\gamma}<+\infty.
	\end{equation}
	We write 
	\begin{equation}\label{E:unitballgamma}
		\mathcal{B}^\gamma_\varrho:=\left\lbrace X\in \mathcal{C}_0^\gamma([0,1]^k,\mathbb{R}^n):\ \|X\|_{\mathcal{C}_0^\gamma}\leq \varrho\right\rbrace
	\end{equation}
	for the closed ball in $\mathcal{C}_0^\gamma$ with center zero and radius $\varrho>0$.

	\section{H\"older-constrained minimal self-interaction}\label{S:minselfinteract}
	
	Our first observation is the existence of functions \eqref{E:function} minimizing the self-interaction energy
	$X\mapsto I^\alpha(\mu_X)$
	under a H\"older constraint. By \eqref{E:occumeas} this energy can be written as
	\[
	I^\alpha(\mu_X)=\int_{[0,1]^k}\int_{[0,1]^k}\frac{ds\,dt}{|X(t)-X(s)|^{n-\alpha}},
	\]
	with the convention that the integrand is $+\infty$ if $X(t)=X(s)$. Given a nonnegative Radon measure $\nu$ on $\mathbb{R}^n$, we similarly  have  $I^\alpha(\mu_X,\nu)=\int_{[0,1]^k}U^\alpha\nu(X(t))\,dt$.

	\begin{theorem}\label{T:minselfinteraction}
		Let $0<\alpha<n$ and $0<\gamma<\frac{k}{n-\alpha}\wedge 1$. 
		\begin{enumerate}
			\item[(i)] For any $\varrho>0$ there is some $X^\ast\in \mathcal{B}^\gamma_\varrho$ minimizing $X\mapsto I^\alpha(\mu_X)$ over $\mathcal{B}^\gamma_\varrho$.
			\item[(ii)] For any minimizer $X^\ast$ of $X\mapsto I^\alpha(\mu_X)$ in $\mathcal{B}^\gamma_\varrho$ and any non-degenerate rectangle $\mathcal{R}\subset [0,1]^k$ the image $X^\ast(\mathcal{R})$ satisfies
			\begin{equation}\label{E:measrange}
				\mathcal{H}^{n-\alpha}(X^\ast(\mathcal{R}))=+\infty,\qquad \mathcal{H}^{\frac{k}{\gamma}\wedge n}(X^\ast(\mathcal{R}))<+\infty
			\end{equation}
			and, as a consequence, 
			\begin{equation}\label{E:dimrange}
				n-\alpha\leq \dim_H X^\ast(\mathcal{R})\leq \frac{k}{\gamma}\wedge n.
			\end{equation}
		\end{enumerate}
	\end{theorem}
	
	\begin{remark}\label{R:trivia}\mbox{}
		\begin{enumerate}
			\item[(i)] We are mainly interested in cases where $1\leq k\leq n$. In such cases $\dim_H X^\ast(\mathcal{R})$ will be close to $n-\alpha$, provided that $\alpha<n-k$ and $\gamma$ is chosen close to its upper bound. More precisely, if $\alpha$ is non-integer and $\frac{k}{\lceil n-\alpha\rceil}<\gamma<\frac{k}{n-\alpha}$, where $\lceil n-\alpha\rceil$ denotes the smallest integer larger than $n-\alpha$, then \eqref{E:dimrange} forces $\dim_H X^\ast(\mathcal{R})$ to be non-integer.
			\item[(ii)] Theorem \ref{T:minselfinteraction} (ii) follows from well-known facts: The finiteness of 
			$I^\alpha(\mu_{X^\ast})$ gives the first item in \eqref{E:measrange} and the lower bound in \eqref{E:dimrange}, the H\"older constraint gives the second item in \eqref{E:measrange} and the upper bound in \eqref{E:dimrange}. See \cite[Proposition 2.3 and Theorem 4.13]{Falconer}.
		\end{enumerate}
	\end{remark}
	
	\begin{remark}
		Although $\mu\mapsto I^\alpha(\mu)$ is quadratic and, since $k_\alpha$ is positive definite, convex, the functional $X\mapsto I^\alpha(\mu_X)$ is not convex. Not even its effective domain $\{X:[0,1]^k\to\mathbb{R}^n\ \text{Borel}:\ I^\alpha(\mu_X)<+\infty\}$ is convex. The functional $X\mapsto I^\alpha(\mu_X)$ is homogeneous in the sense that 
		\begin{equation}\label{E:homo}
			I^\alpha(\mu_{\lambda X})=\lambda^{\alpha-n}I^\alpha(\mu_X),\qquad \lambda>0.
		\end{equation}
	\end{remark}
	
	\begin{remark}\mbox{}
		\begin{enumerate}
			\item[(i)] One cannot expect $X^\ast$ to be unique; the radial symmetry of $k_\alpha$ gives counterexamples.
			\item[(ii)] It is well known that given a compact set $K\subset\mathbb{R}^n$ of positive $\alpha$-capacity, the energy $\mu\mapsto I^\alpha(\mu)$ admits a unique minimizer within
			the class of all Borel probability measures $\mu$ supported in $K$, namely the equilibrium measure $\mu^K$ of $K$ with respect to $I^\alpha$; see \cite[Chapter II, Section 1]{Landkof}. For $K=\overline{B}(0,\varrho k^{\gamma/2})$ we clearly have $I^\alpha(\mu^K)\leq I^\alpha(\mu_{X^\ast})$. The classical equilibrium problem does not involve any further restriction on the support of $\mu^K$, while $\mu_{X^\ast}$ has to be an occupation measure and, in particular, to obey the topological constraint \eqref{E:contcasesupport}, that is, $\supp\mu_{X^\ast}=X^\ast([0,1]^k)$.
			\item[(iii)] For self-interaction energies $\mu\mapsto\int_{\mathbb{R}^n}\int_{\mathbb{R}^n} W(x-y)\mu(dy)\mu(dx)$, viewed as functionals on Borel probability measures $\mu$ on $\mathbb{R}^n$, (local) minimizers may actually be more diffuse: If, roughly speaking, $-\Delta W(x)\geq c|x|^{-\beta}$ for small $x$, $\Delta W$ is singular at the origin $0$ and a number of further natural hypotheses are satisfied, then local minimizers $\mu_\ast$ with respect to the $\infty$-Wasserstein distance have a support $\supp\mu_\ast$ with Hausdorff dimension at least $\beta$. This was proved in \cite[Theorem 1]{Balague2013}. For $W=k_\alpha$ and $2<\alpha<n$ this gives $\beta=n-\alpha+2$, which is a better lower bound on $\dim_H \supp \mu_\ast$ than the generic $n-\alpha$ in \eqref{E:dimrange}. On the other hand, the authors of \cite{Balague2013} comment that simulations never showed a non-integer Hausdorff dimension for $\supp \mu_\ast$, see \cite[p.~1058]{Balague2013}.
		\end{enumerate}
	\end{remark}
	
	Theorem \ref{T:minselfinteraction} (i) follows from the classical direct method and two observations. The first observation is a simple but useful compactness result.
	
	\begin{lemma}\label{L:compact}
		Every sequence $(X_i)_i\subset \mathcal{B}_\varrho^\gamma$ has a subsequence $(X_{i_j})_j$ converging uniformly to some $X^\ast\in\mathcal{B}_\varrho^\gamma$ and such that $(\mu_{X_{i_j}})_j$ converges weakly to $\mu_{X^\ast}$.
	\end{lemma}
	
	\begin{proof}
		The first claim holds by the Arzel\`a--Ascoli theorem; the limit $X^\ast$ lies in $\mathcal{B}_\varrho^\gamma$, since the conditions $X_{i_j}(0)=0$ and $|X_{i_j}(t)-X_{i_j}(s)|\leq \varrho|t-s|^\gamma$, $s,t\in[0,1]^k$, pass to pointwise limits. Since $\mathcal{L}^k([0,1]^k)=1$, uniform convergence implies convergence in $\mathcal{L}^k|_{[0,1]^k}$-measure. As a consequence, the occupation measures $\mu_{X_{i_j}}$ converge weakly to $\mu_{X^\ast}$, see \cite[Lemma 5.7]{Kallenberg} or \cite[p.~150]{ABM}.
	\end{proof}
	
	The second observation is that $\mathcal{B}_\varrho^\gamma$ contains elements $X$ with $I^\alpha(\mu_X)<+\infty$.
	
	\begin{lemma}\label{L:Ialphafinite}
		Let $0<\alpha<n$ and $0<\gamma<\frac{k}{n-\alpha}\wedge 1$. For every $\varrho>0$ there is some $X\in \mathcal{B}^\gamma_\varrho$ with $I^\alpha(\mu_X)<+\infty$. 
	\end{lemma}

	\begin{proof}[Proof of Theorem \ref{T:minselfinteraction} (i)]
		Note first that $\inf_{X\in \mathcal{B}^\gamma_{\varrho}}I^\alpha(\mu_{X})<+\infty$ by Lemma \ref{L:Ialphafinite}. Let $(X_i)_i\subset \mathcal{B}^\gamma_{\varrho}$ be such that $\lim_{i\to \infty} I^\alpha(\mu_{X_i})=\inf_{X\in \mathcal{B}^\gamma_{\varrho}}I^\alpha(\mu_{X})$. By Lemma \ref{L:compact} there are a subsequence $(X_{i_j})_j$ and some $X^\ast\in \mathcal{B}^\gamma_{\varrho}$ such that $(X_{i_j})_j$ converges uniformly to $X^\ast$ and $(\mu_{X_{i_j}})_j$ converges weakly to $\mu_{X^\ast}$. Since $\mu\mapsto I^\alpha(\mu)$ is lower semicontinuous with respect to weak convergence \cite[p.~78]{Landkof}, we have
		\begin{equation*}
			\inf_{X\in \mathcal{B}^\gamma_{\varrho}}I^\alpha(\mu_{X})\leq I^\alpha(\mu_{X^\ast})\leq \liminf_{j\to\infty}I^\alpha(\mu_{X_{i_j}})=\inf_{X\in \mathcal{B}^\gamma_{\varrho}}I^\alpha(\mu_{X}).
		\end{equation*} 
	\end{proof}
	
	\begin{remark}\label{R:polymers} For $k=1$ the minimizer $X^\ast$ is a curve with image $X^\ast([0,1])\subset \mathbb{R}^n$. For the physically relevant case $n=3$ and $\alpha=2$ the finiteness of $I^2(\mu_{X^\ast})$ forces $X^\ast([0,1])$ to be nowhere rectifiable in the sense that any arc $X^\ast([a,b])$, $0\leq a<b\leq 1$, has infinite length $\mathcal{H}^1(X^\ast([a,b]))=+\infty$. Such curves might be candidates for models of individual charged polymer chains, as already suggested in  \cite{Havlin1982a, Havlin1982b}.
	\end{remark}

	Lemma \ref{L:Ialphafinite} can be proved in several ways. Under the additional restriction that 
	\begin{equation}\label{E:Assouadcond}
		k(\lfloor \frac{1}{\gamma}\rfloor+1)\leq n<\frac{k}{\gamma}+\alpha,
	\end{equation}
	where $\lfloor \frac{1}{\gamma}\rfloor$ denotes the integer part of $\frac{1}{\gamma}$, the most immediate proof 
	is provided by Assouad's embedding theorem \cite[Proposition 4.4]{Assouad83}, which states that, under the first inequality in \eqref{E:Assouadcond}, there exist a function $X:[0,1]^k\to\mathbb{R}^n$ and constants $C>c>0$ satisfying \eqref{E:Assouad}. The second inequality in \eqref{E:Assouadcond} is equivalent to $\gamma<\frac{k}{n-\alpha}$, which is assumed in Lemma \ref{L:Ialphafinite}; it guarantees that $I^\alpha(\mu_X)<+\infty$. We may assume that $X(0)=0$ and $C\leq\varrho$. Condition \eqref{E:Assouadcond} is strict in the sense that even for $k=1$ the integer $\lfloor \frac{1}{\gamma}\rfloor+1$ in the left-hand side cannot be replaced by $\lfloor \frac{1}{\gamma}\rfloor$, see \cite[Section 4.5]{Assouad83}. Condition \eqref{E:Assouadcond} restricts the choice of possible $\alpha$; for instance, $k=1$ and $\gamma=\frac12$ require $\alpha>1$. Another proof for Lemma \ref{L:Ialphafinite} in the case $k=1$ is given in \cite[Theorem 34]{GG20b}, which shows that functions $X$ with the desired properties are abundant. The proof of \cite[Theorem 34]{GG20b} is essentially probabilistic, see \cite[Theorem 33]{GG20b}. To prove Lemma \ref{L:Ialphafinite} in its full generality we can use a similar but simpler probabilistic argument. 
	
	Recall that a \emph{fractional Brownian $(k,1)$-field $b^H$ with Hurst index $0<H<1$} over a probability space $(\Omega,\mathcal{F},\mathbb{P})$ is a centered Gaussian random field $b^H:[0,+\infty)^k\times\Omega\to \mathbb{R}$ satisfying $b^H(0)=0$ $\mathbb{P}$-a.s. and 
	$\mathbb{E}|b^H(t)-b^H(s)|^2=|t-s|^{2H}$, $s,t\in [0,+\infty)^{k}$.
	The case $k=1$ gives fractional Brownian motion with Hurst index $0<H<1$, the case $H=\frac12$ gives L\'evy's Brownian field; Brownian motion results if both $k=1$ and $H=\frac12$. Note that, since the increments are centered Gaussian, 
	\begin{equation}\label{E:genGauss}
		\mathbb{E}|b^H(t)-b^H(s)|^\ell=(\ell-1)!!|t-s|^{\ell H},\qquad s,t\in [0,+\infty)^{k},
	\end{equation}
	for any even integer $\ell\geq 2$. Given $n$ independent copies $b_1^H,\dots,b_n^H$ of $b^H$, the random field $B^H=(b_1^H,\dots,b_n^H):[0,+\infty)^k\times\Omega\to \mathbb{R}^n$ is called a \emph{fractional Brownian $(k,n)$-field $B^H$ with Hurst index $0<H<1$}. It is centered, Gaussian, satisfies $B^H(0)=0$ $\mathbb{P}$-a.s. and 
	\begin{equation}\label{E:indepcomp}
		\mathbb{E}|B^H(t)-B^H(s)|^\ell\le n^{\ell/2}\:\mathbb{E}|b^H(t)-b^H(s)|^\ell,\quad s,t\in [0,+\infty)^k,
	\end{equation}
	for any even integer $\ell\geq 2$. See for instance \cite[Chapter 18]{Kahane} or \cite{Xiao06}.
	
	\begin{proof}[Proof of Lemma \ref{L:Ialphafinite}]
		Let $0<\gamma<H<\frac{k}{n-\alpha}\wedge 1$ and let $B^H$ be a fractional Brownian $(k,n)$-field with Hurst index $H$ over $(\Omega,\mathcal{F},\mathbb{P})$. The Kolmogorov--Chentsov theorem \cite[Theorem 1.4.1]{Kunita90} together with \eqref{E:genGauss} and \eqref{E:indepcomp} ensures the existence of a modification of $B^H$, a random variable $K>0$ and an event $\Omega_0\in \mathcal{F}$ with $\mathbb{P}(\Omega_0)=1$ such that
		$|B^H(t,\omega)-B^H(s,\omega)|\leq K(\omega)|t-s|^\gamma$ for all $s,t\in [0,1]^k$ and $\omega\in \Omega_0$.
		Proceeding as in \cite[Example 4.23]{HTV2020}, we see that $\mathbb{E}|B^H(t)-B^H(s)|^{\alpha-n}=c|t-s|^{(\alpha-n)H}$;
		here $c$ is a positive constant independent of $s$ and $t$. Since $(n-\alpha)H<k$, this gives
		\begin{equation*}
			\mathbb{E}\, I^{\alpha}(\mu_{B^H})=\int_{[0,1]^k}\int_{[0,1]^k}\mathbb{E}|B^H(t)-B^H(s)|^{\alpha-n}\,ds\,dt=c\int_{[0,1]^k}\int_{[0,1]^k}\frac{ds\,dt}{|t-s|^{(n-\alpha)H}}<+\infty.
		\end{equation*}
		Consequently we can find $\Omega_1\in\mathcal{F}$ with $\mathbb{P}(\Omega_1)=1$ and $\Omega_1\subset \Omega_0$ such that $I^{\alpha}(\mu_{B^H(\cdot,\omega)})<+\infty$ for all $\omega\in \Omega_1$. With such $\omega\in \Omega_1$ we can now, by \eqref{E:homo}, take $X(t):=\varrho(1+K(\omega))^{-1}B^H(t,\omega)$, $t\in [0,1]^k$.
	\end{proof}

	\section{Bounded potentials and constrained minimal interaction}\label{S:velocity}
	
	Suppose that $\nu$ is a given Borel probability measure on $\mathbb{R}^n$.  A further question is whether the minimization of the nonlinear functional $X\mapsto I^\alpha(\mu_X,\nu)$
	under a H\"older condition (\ref{E:Hoelder}) will force minimizers to have an image of non-integer Hausdorff dimension.
	
	In general this is not the case: The finiteness of $I^\alpha(\mu_X,\nu)$ does not necessarily have any implication for the geometry of the image $X([0,1]^k)$. This is easily seen in the case $k=1$.
	
	\begin{example}\label{Ex:interval}
		Let $n=2$, $0<\alpha<1$, $\nu=\mathcal{H}^1|_{\{0\}\times [-1,1]}$ (here $\nu$ is not normalized, which is irrelevant) and let $X=(X_1,X_2):[0,1]\to [0,+\infty)\times \mathbb{R}$ be a curve with $X(0)=0$ and $X_1(t)>0$, $0<t\leq 1$. Since 
		\begin{align}
			\int_{\mathbb{R}^2}\int_{\mathbb{R}^2}|x-y|^{\alpha-2}\nu(dy)\mu_X(dx)
			&=\int_0^1\int_{-1}^1 \left[X_1(t)^2+(y-X_2(t))^2\right]^{\frac{\alpha}{2}-1}\:dy\:dt \notag\\
			&=\int_0^1 \int_{-1}^1 \left[X_1(t)^2\left(1+\Big(\frac{y-X_2(t)}{X_1(t)}\Big)^2\right)\right]^{\frac{\alpha}{2}-1}\:dy\:dt\notag\\
			&\leq \Big(\int_{-\infty}^{+\infty}(1+\eta^2)^{\frac{\alpha}{2}-1}d\eta\Big)\int_0^1X_1(t)^{\alpha-1}\:dt,\notag
		\end{align}
		a sufficient condition for the finiteness of $I^\alpha(\mu_X,\nu)$ is to have $X_1(t)\geq c\:t^{b}$ with some $c>0$ and $0<b<\frac{1}{1-\alpha}$. This is satisfied for the Lipschitz curve $(X_1(t),X_2(t))=(t,0)$, $t\in [0,1]$.
	\end{example}

	If $\nu$ is very concentrated, the curve $X$ may still be a parametrization of a line segment, but it has to leave the support of $\nu$ more quickly. 
	
	\begin{example}\label{Ex:pointmass}
		Let $n=2$, $0<\alpha<1$, let $\nu=\delta_0$ be the unit mass at the origin and let $X:[0,1]\to [0,+\infty)\times\mathbb{R}$ be a curve with $X(0)=0$. Then 
		\begin{equation}\label{E:pointmass}
			\int_{\mathbb{R}^2}\int_{\mathbb{R}^2}|x-y|^{\alpha-2}\nu(dy)\mu_X(dx)=\int_0^1 |X(t)|^{\alpha-2}dt
		\end{equation}
		cannot be finite if $X$ is Lipschitz. However, even if $X([0,1])$ is a line segment, (\ref{E:pointmass}) is finite if $X$ leaves $0$ fast enough; an example is the curve $X(t)=(t^\gamma,0)$, $t\in [0,1]$ with $0<\gamma<\frac{1}{2-\alpha}$, which has \enquote{infinite velocity} at the initial point.
	\end{example} 
	
	For minimizers of $X\mapsto I^\alpha(\mu_X,\nu)$ to have an image of non-integer dimension, constraints on the \enquote{velocity} of $X$ are needed. Since we are interested in non-absolutely continuous curves $X:[0,1]\to\mathbb{R}^n$, no classical variant of this concept is available. Overly abrupt changes of \enquote{velocity} could be excluded through bounds on the asymptotic uniform behavior of the oscillation type quantity $\sup_{0<|t-s|\leq h}|X(t)-X(s)|$. We call a positive increasing function $m_+$, regularly varying at $0$, a \emph{uniform upper modulus for $X$} if
	\begin{equation}\label{E:uppermod}
		\limsup_{h\to 0} \frac{1}{m_+(h)}\sup_{h\leq t\leq 1-h}\ \sup_{0<|t-s|\leq h}|X(t)-X(s)|<+\infty
	\end{equation}
	holds, and a positive increasing function $m_-$, regularly varying at $0$, a \emph{uniform lower modulus for $X$} if
	\begin{equation}\label{E:lowermod}
		\liminf_{h\to 0}\frac{1}{m_-(h)}\inf_{h\leq t\leq 1-h}\ \sup_{0<|t-s|\leq h}|X(t)-X(s)|>0
	\end{equation}
	holds.
	
	Recall that a function $m:(0,\varepsilon)\to (0,+\infty)$, where $\varepsilon>0$, is \emph{regularly varying at $0$ with index $\kappa\in\mathbb{R}$} if 
	$\lim_{h\to 0}\frac{m(\lambda h)}{m(h)}=\lambda^\kappa$ for all $\lambda>0$,
	and that $m$ is said to be \emph{slowly varying at $0$} if it is regularly varying at $0$ with index $\kappa=0$.  See \cite[Definition 1.1 and remarks following it]{Seneta76}. A function $m$ is regularly varying at $0$ with index $\kappa$ if and only if $m(h)=h^\kappa\psi(h)$ with some $\psi$ slowly varying at $0$. 
	
	A suitable condition, weaker than the bi-H\"older condition \eqref{E:Assouad} but similar in spirit, could be to require that $X$ has moduli $m_+$ and $m_-\leq m_+$ as in (\ref{E:uppermod}) and (\ref{E:lowermod}) with the same index $\kappa>0$ of regular variation. Since we are interested in non-Lipschitz curves, we could assume that $\kappa<1$.

	\begin{examples}\label{Ex:Brownian}\mbox{}
		\begin{enumerate}
			\item[(i)] Assouad's Koch curve $X$ in \eqref{E:Assouad} satisfies (\ref{E:uppermod}) and (\ref{E:lowermod}) with $m_+(h)=m_-(h)=h^\gamma$.
			\item[(ii)] For typical Brownian paths the simultaneous validity of conditions (\ref{E:uppermod}) and (\ref{E:lowermod}) with
			$m_+(h)=h^{1/2}(-\log h)^{1/2}$ and $m_-(h)=h^{1/2}(-\log h)^{-1/2}$ 
			is ensured by L\'evy's theorem on the modulus of continuity (for $m_+$) and by the Cs\"org\Humlaut{o}--R\'ev\'esz theorem on the modulus of non-differentiability (for $m_-$), see \cite{CsorgoRevesz79, Levy37} or \cite[Theorems 1.1.1 and 1.6.1]{CsorgoRevesz81}. Analogous results for fractional Brownian motions with Hurst index $H\in (0,1)$ can be found in \cite[Theorem 7.6.8]{MarcusRosen06} and \cite[Theorem 1.2]{WangXiao19}, respectively.
		\end{enumerate}
	\end{examples}
	
	\begin{remark}\mbox{}
		\begin{enumerate}
			\item[(i)] Suppose that $X=(X(t))_{t\geq 0}$ is a stochastic process $X:[0,+\infty)\times \Omega\to \mathbb{R}$ over some probability space $(\Omega,\mathcal{F},\mathbb{P})$ such that $\mathbb{E}|X(t)-X(s)|^2 \leq c|t-s|^{2H}$, $0\leq s<t<+\infty$, with constants $c>0$ and $H\in(0,1)$ and having the hypercontractivity property 
			$\mathbb{E}|X(t)-X(s)|^p \leq c_0^p p^{\iota p} \left[\mathbb{E}|X(t)-X(s)|^2\right]^{p/2}$, $0\leq s<t<+\infty$,
			for all $p\geq 1$, where $c_0>0$ and $\iota \geq 0$ are constants independent of $p$. Then almost surely the paths of $X$ satisfy \eqref{E:uppermod} with 
			$m_+(h) = h^H\left(-\log h\right)^\iota$,
			see \cite[Corollary 2.11]{Nummi-Viitasaari}. These properties hold for a rich class of Gaussian processes, for processes living in a fixed Wiener chaos (such as Hermite processes) and for solutions to SDEs driven by fractional Brownian motion.
			\item[(ii)] Condition \eqref{E:lowermod} can often be shown using the upper regularity of the local time, provided it exists. See for instance \cite[Corollary 2.6]{Sonmez-Sottinen-V}, which states that in many cases of interest one can choose
			$m_-(h) = h^{H}(-\log h)^{-H(1+\theta)}$
			with suitable $\theta$. Such cases include locally nondeterministic Gaussian processes satisfying $\mathbb{E}|X(t)-X(s)|^2 \leq c|t-s|^{2H}$, the Rosenblatt process of order $H$, and solutions of SDEs driven by a fractional Brownian motion; see \cite[Theorem 2.7 and Propositions 2.10 and 2.11]{Sonmez-Sottinen-V}. Since these results rely on the existence of local times, they apply to $\mathbb{R}^n$-valued processes with Hurst index $H$ only if $Hn<1$. For instance, an $n$-dimensional fractional Brownian motion with Hurst index $H$ has local times if and only if $Hn<1$.
			\item[(iii)] Lower bounds on the oscillation of paths in the spirit of \eqref{E:lowermod} have been proposed in \cite{FrizShekhar} and \cite[Chapter 6]{FrizHairer}, albeit only with polynomial moduli and without slowly varying corrections. The notion in \cite[Definition 1]{FrizShekhar} is a pointwise rather than a uniform condition.			
		\end{enumerate}
	\end{remark}

	The condition on a curve to have moduli $m_+$ and $m_-$ of the same index $\kappa$ seems too fragile to be used in variational arguments. We replace (\ref{E:uppermod}) by a conventional H\"older condition as in (\ref{E:Hoelder}). If the curve $X$ is $\gamma$-H\"older continuous, then it satisfies (\ref{E:uppermod}) with any modulus $m_+$ with index $\kappa_+<\gamma$. We relax (\ref{E:lowermod}) by instead requiring the boundedness of the $\alpha$-Riesz potential $U^\alpha\mu_X$ as in \eqref{E:Rieszpot}. This ensures (\ref{E:lowermod}) for any modulus $m_-$ with index $\kappa_->\frac{1}{n-\alpha}$.
	
	\begin{lemma}\label{L:bdpot} Let $X:[0,1]\to\mathbb{R}^n$ be a  curve and let $0<\alpha<n$. If 
		\begin{equation}\label{E:bdpot}
			\sup_{x\in X([0,1])} U^\alpha \mu_X(x)<+\infty,
		\end{equation}
		then (\ref{E:lowermod}) holds for any regularly varying $m_-$ with index $\frac{1}{n-\alpha}<\kappa_-$.
	\end{lemma}
	
	\begin{remark}\label{R:Landkof}\mbox{}
		\begin{enumerate}
			\item[(i)] Condition \eqref{E:bdpot} implies that $I^\alpha(\mu_X)<+\infty$, but it is strictly stronger.
			\item[(ii)] Taking the supremum in (\ref{E:bdpot}) over $\mathbb{R}^n$ instead of $X([0,1])$ does not change the condition \cite[Section I.3, Theorem 1.5]{Landkof} (the supremum over $\mathbb{R}^n$ is at most $2^{n-\alpha}$ times the original one).
		\end{enumerate}
	\end{remark}
	
	\begin{remark}\label{R:Brownian}\mbox{}
		\begin{enumerate}
			\item[(i)] Obviously \eqref{E:bdpot} holds for Assouad's Koch curve $X$ in \eqref{E:Assouad} if $\gamma<\frac{1}{n-\alpha}$.
			\item[(ii)] Let $n\geq 3$, $k=1$, $n-2<\alpha<n$ and $0<\gamma<\frac12$. Typical paths $t\mapsto W(t,\omega)$ of a Brownian motion $W$ in $\mathbb{R}^n$ are $\gamma$-H\"older and satisfy 
			$\sup_{x\in W(\cdot,\omega)([0,1])}U^{\alpha}\mu_{W(\cdot,\omega)}(x)<+\infty$.
			The latter can be seen using the fact that there is a constant $c>0$ such that
			$\limsup_{r\to 0}\sup_{x\in W([0,1])}\frac{\mu_{W}(B(x,r))}{r^2|\log r|}<c$ a.s.
			This was proved in \cite[Lemma 2.3]{PerkinsTaylor87} and, in a refined version, in \cite[Theorem 1.3]{Demboetal00}. Such results have been extended to cover the case of stable-like processes, see \cite[Theorem 1.5]{Seuret-Yang}. See also \cite[Theorem 1.1]{Lin-Wang} for the case of Brownian sheets.
		\end{enumerate}
	\end{remark}
	
	\begin{proof}[Proof of Lemma \ref{L:bdpot}]
		Let $\kappa_-$ be as stated. We may assume that $m_-(h)=h^{\kappa_-}\psi(h)$ with $\psi$ slowly varying at $0$. Suppose that (\ref{E:lowermod}) does not hold, 
		\[\liminf_{h\to 0}\frac{1}{m_-(h)}\inf_{h\leq t\leq 1-h}\ \sup_{0<|t-s|\leq h}|X(t)-X(s)|=0.\]
		Then, given $\varepsilon>0$, we can find a sequence $(h_j)_j$ converging to zero as $j\to\infty$ such that for all sufficiently large $j$ we have 
		\[\inf_{h_j\leq t\leq 1-h_j}\sup_{0<|t-s|\leq h_j}|X(t)-X(s)|<\varepsilon h_j^{\kappa_-} \psi(h_j).\]
		This gives 
		\begin{align}
			\sup_{h_j\leq t\leq 1-h_j}\int_{t-h_j}^{t+h_j}|X(t)-X(s)|^{\alpha-n}ds &\geq 2h_j\sup_{h_j\leq t\leq 1-h_j}\inf_{0<|t-s|\leq h_j}|X(t)-X(s)|^{\alpha-n}\notag\\
			&> 2\varepsilon^{\alpha-n}h_j^{\kappa_-(\alpha-n)+1}\psi(h_j)^{\alpha-n}.\notag
		\end{align}
		Since the right-hand side goes to $+\infty$ as $j\to\infty$ and the left-hand side is bounded above by 
		\[\sup_{t\in [0,1]}\int_0^1 |X(t)-X(s)|^{\alpha-n}ds=\sup_{x\in X([0,1])} U^\alpha \mu_X(x),\]
		this contradicts (\ref{E:bdpot}).
	\end{proof}
	
	These considerations suggest looking at curves $X$ satisfying (\ref{E:bdpot}) and being $\gamma$-H\"older continuous with $0<\gamma<\frac{1}{n-\alpha}$. Then \eqref{E:uppermod} and \eqref{E:lowermod} would be satisfied with any moduli $m_+$ and $m_-$ having indices $\kappa_+$ and $\kappa_-$ such that $0<\kappa_+< \gamma< \frac{1}{n-\alpha}<\kappa_-$. Fixing $\gamma$ and $\frac{1}{n-\alpha}$ could be considered as fixing a \enquote{velocity band} (or window) for such curves.

	As before, let $\nu$ be a Borel probability measure on $\mathbb{R}^n$. We now return to general parameter dimensions $k\geq 1$. Motivated by the preceding discussion, we aim to minimize $X\mapsto I^\alpha(\mu_X,\nu)$
	under a H\"older condition (\ref{E:Hoelder}) and an additional
	bound on $U^\alpha \mu_X$.
	Given $\varrho>0$ and $M>0$, let 
	\[\mathcal{K}(\gamma,\alpha,\varrho,M):=\Big\lbrace X\in \mathcal{B}^\gamma_\varrho:\quad \sup_{x\in \mathbb{R}^n} U^\alpha\mu_X(x)\le M\Big\rbrace.\]
	 By Remark \ref{R:Landkof} (ii), taking the supremum over $\mathbb{R}^n$ is not fundamentally different from \eqref{E:bdpot}.
	
	\begin{theorem}\label{T:mininteraction}
		Let $0<\alpha< n$, $0<\gamma<\frac{k}{n-\alpha}\wedge 1$ and $\varrho>0$. Let $\nu$ be a Borel probability measure on $\mathbb{R}^n$ and $M>0$.
		\begin{enumerate}
			\item[(i)] If $\mathcal{K}(\gamma,\alpha,\varrho,M)\neq \emptyset$, then there is some $X^\ast\in \mathcal{K}(\gamma,\alpha,\varrho,M)$ minimizing $X\mapsto I^\alpha(\mu_X,\nu)$ over $\mathcal{K}(\gamma,\alpha,\varrho,M)$.
			\item[(ii)] For any minimizer $X^\ast$ of $X\mapsto I^\alpha(\mu_X,\nu)$ in $\mathcal{K}(\gamma,\alpha,\varrho, M)$ and any non-degenerate rectangle $\mathcal{R}\subset [0,1]^k$ the image $X^\ast(\mathcal{R})$ satisfies (\ref{E:measrange}) and (\ref{E:dimrange}). 
			\item[(iii)] There are $\varrho_1>0$ and $M_1>0$, depending on $n$, $k$, $\gamma$ and $\alpha$, such that $\mathcal{K}(\gamma,\alpha,\varrho,M)\neq \emptyset$ whenever $\varrho>\varrho_1$ and $M>M_1$. 
		\end{enumerate}
	\end{theorem}
	
	For $0<\gamma<\frac{k}{2(n-\alpha)}\wedge 1$ explicit expressions for possible $\varrho_1$ and $M_1$ can be given, see Remark \ref{R:explicit} below.

	\begin{remark}\mbox{}
		\begin{enumerate} 
			\item[(i)] Theorem \ref{T:mininteraction} is interesting if the supports of $\nu$ and $\mu_X$ intersect, for instance, if $0\in \supp\nu$, as in Examples \ref{Ex:interval} and \ref{Ex:pointmass}. Since \eqref{E:contcasesupport} implies that $\supp\mu_X\subset \overline{B}(0,\varrho k^{\gamma/2})$ for all $X\in \mathcal{B}^\gamma_\varrho$, the most natural examples are measures $\nu$ with compact support; this assumption, however, is not needed in Theorem \ref{T:mininteraction}.
			\item[(ii)] In the case $k=1$ Theorem \ref{T:mininteraction} says that within the class $\mathcal{K}(\gamma,\alpha,\varrho,M)$ of H\"older curves \enquote{whose velocity does not vary too wildly} there are curves minimizing the mutual interaction energy with the given medium $\nu$.
			\item[(iii)] Clearly the geometric properties of $X^\ast(\mathcal{R})$ stated in Theorem \ref{T:mininteraction} (ii) are consequences of the constraints and valid for all elements of $\mathcal{K}(\gamma,\alpha,\varrho, M)$.
		\end{enumerate}
	\end{remark}

	The existence of minimizers in Theorem \ref{T:mininteraction} (i) follows by the direct method \cite[Lemma 4.3.1 and Theorem 4.3.1]{JostLiJost}. We quote a simple case of a well-known result.
	
	\begin{lemma}\label{L:classiclsc}
		Let $E\subset \mathbb{R}^k$ be a Borel set with $\mathcal{L}^k(E)>0$ and let $1\leq q<+\infty$.
		\begin{enumerate}
			\item[(i)] If $f:\mathbb{R}^n\to [0,+\infty]$ is lower semicontinuous, then $\Phi(u):=\int_Ef(u(t))\:dt$
			defines a lower semicontinuous functional $\Phi:L^q(E,\mathbb{R}^n)\to [0,+\infty]$.
			\item[(ii)] If $D\subset L^q(E,\mathbb{R}^n)$ is compact, $f$ and $\Phi$ are as in (i) and
			$\inf_{u\in D}\Phi(u)<+\infty$,
			then there is some $u^\ast\in D$ such that $\Phi(u^\ast)=\min_{u\in D}\Phi(u)$. 
		\end{enumerate}
	\end{lemma}

	\begin{proof}[Proof of Theorem \ref{T:mininteraction}] We will write $\mathcal{K}:=\mathcal{K}(\gamma,\alpha,\varrho,M)$ to shorten notation.
		The function $y\mapsto U^\alpha\nu(y)$ is nonnegative and lower semicontinuous \cite[7. in Section I.3]{Landkof} on $\mathbb{R}^n$. By Lemma \ref{L:classiclsc} (i) the functional 
		$X\mapsto I^\alpha(\mu_X,\nu)=\int_{[0,1]^k}U^\alpha\nu(X(t))dt$
		is lower semicontinuous on $L^1([0,1]^k,\mathbb{R}^n)$. Moreover, $I^\alpha(\mu_X,\nu)\leq M$ 
		for any $X \in \mathcal{K}$. To see that $\mathcal{K}$ is a compact subset of $L^1([0,1]^k,\mathbb{R}^n)$ it suffices, in view of the Arzel\`a--Ascoli theorem and Lemma \ref{L:compact}, to recall that potentials, evaluated at a fixed point, are lower semicontinuous with respect to the weak convergence of measures \cite[p.~62]{Landkof}. For a sequence $(X_j)_j\subset \mathcal{K}$ with uniform limit $X_\infty$ and such that $\lim_j \mu_{X_j}=\mu_{X_\infty}$ weakly, we have 
		\begin{equation}\label{E:potlsc}
			U^\alpha\mu_{X_\infty}(x)\leq \liminf_{j\to\infty} U^\alpha\mu_{X_j}(x)\leq M,\qquad x\in \mathbb{R}^n.
		\end{equation}
		Since $\mathcal{K}\neq \emptyset$, we have $\inf_{X\in \mathcal{K}}I^\alpha(\mu_X,\nu)<+\infty$, so that Lemma \ref{L:classiclsc} gives item (i) in Theorem \ref{T:mininteraction}. Item (ii) is as before. Under the additional condition \eqref{E:Assouadcond} item (iii) is immediate from \eqref{E:Assouad}. In the general case it follows from an application of
		Lemma \ref{L:BermanPitt} (i) below with a fractional Brownian $(k,n)$-field $B^H$ with Hurst index $\gamma<H<\frac{k}{n-\alpha}\wedge 1$.
	\end{proof}
	
	\begin{remark}\label{R:explicit}\mbox{}
		\begin{enumerate}
			\item[(i)] In the case $0<\gamma<\frac{k}{2(n-\alpha)}\wedge 1$ one can find explicit expressions for possible $\varrho_1$ and $M_1$ in Theorem \ref{T:mininteraction} (iii): Consider a fractional Brownian $(k,n)$-field $B^H$ with Hurst index $\gamma<H<\frac{k}{2(n-\alpha)}\wedge 1$, choose an even integer $\ell\geq 2$ large enough to have $\gamma<H-\frac{2k}{\ell}$ and apply Lemma \ref{L:Bermanexplicit} (ii). Then, given any $M^\ast$ larger than the right-hand side of \eqref{E:M1}, $M_1:=M^\ast$ and $\varrho_1:=k^{(H-\gamma)/2-k/\ell}\,c(k,H-\frac{k}{\ell},\ell)\,(M^\ast)^{1/\ell}$ are possible choices.
			\item[(ii)] If \eqref{E:Assouadcond} holds, possible $\varrho_1$ and $M_1$ could be extracted from \cite[Section 3]{Assouad83}.
			\item[(iii)] For $\lambda>0$ we have
				\begin{equation*}
					U^\alpha\mu_{\lambda X}(x)=\lambda^{\alpha-n}\,U^\alpha\mu_X(x/\lambda),\qquad x\in\mathbb{R}^n,
				\end{equation*}
				so that $\lambda X\in\mathcal{K}(\gamma,\alpha,\lambda\varrho,\lambda^{\alpha-n}M)$ for any $X\in\mathcal{K}(\gamma,\alpha,\varrho,M)$. Consequently, the conditions $\varrho>\varrho_1$ and $M>M_1$ in Theorem \ref{T:mininteraction} (iii) can be replaced by the single condition $\varrho^{n-\alpha}M>\varrho_1^{n-\alpha}M_1$.
		\end{enumerate}
	\end{remark}
	
	\begin{remark}\label{R:nonlin}
		The same arguments give an immediate generalization of Theorem \ref{T:mininteraction} to functionals $X\mapsto \int_{\mathbb{R}^n}(U^\alpha\mu_X)^pd\nu$ with $1\leq p<+\infty$.
	\end{remark}
	
	\section{Potentials of Gaussian occupation measures}\label{S:occupot}
	
	We prove results on the continuity and boundedness of potentials of occupation measures of Gaussian random fields. They are natural analogues of 
	classical results for local times. In Lemma \ref{L:Bermanbasic} we state variants of some of Berman's results \cite{Berman69}, solely based on Fourier analysis. The main result of this section is Theorem \ref{T:Pitt}, which is parallel to Pitt's \cite[Theorem 4 and Proposition 7.3]{Pitt78}, based on local nondeterminism.
	
	\begin{remark}
		If a continuous function $X:[0,1]^k\to \mathbb{R}^n$ has local times $L_X\in L^p_{\loc}(\mathbb{R}^n)$ for some $1< p\leq \infty$ with $\alpha p>n$, then $U^\alpha\mu_X$ is bounded. This follows from H\"older's inequality; related arguments are given in \cite[Proposition 4.14]{HTV2020}. In particular, the local boundedness of $L_X$ implies the boundedness of $U^\alpha\mu_X$ for all $0<\alpha<n$.
	\end{remark}
	
	We denote by $C_b(\mathbb{R}^n)$ the space of bounded continuous functions on $\mathbb{R}^n$ and by $C_\infty(\mathbb{R}^n)$ the space of continuous functions on $\mathbb{R}^n$ vanishing at infinity. Given a Borel measure $\mu$ on $\mathbb{R}^n$, 
	$\hat{\mu}(\xi):=\int_{\mathbb{R}^n} e^{i\left\langle x,\xi\right\rangle}\mu(dx)$, $\xi\in \mathbb{R}^n$,
	denotes its Fourier transform. For finite $\mu$ we clearly have $\hat{\mu}\in C_b(\mathbb{R}^n)$. Given $0<\alpha<n$, let $c_{n,\alpha}:=(2\pi)^{-n}2^\alpha\pi^{n/2}\Gamma(\frac{\alpha}{2})/\Gamma(\frac{n-\alpha}{2})$; then the Fourier transform of $k_\alpha$ is $(2\pi)^n c_{n,\alpha}|\xi|^{-\alpha}$.

	In \cite[Corollaries 2.3 and 3.2]{HTV2021-1} we had listed basic Sobolev regularity results for occupation measures $\mu_X$ and local times $L_X$. We complement these results by a straightforward generalization of \cite[Lemma 2.2]{Berman69}. 
	
	\begin{lemma}
		Let $\mu$ be a finite nonnegative Borel measure on $\mathbb{R}^n$.
		\begin{enumerate}
			\item[(i)] If $0< \alpha<n$ and $\int_{\mathbb{R}^n}|\hat{\mu}(\xi)||\xi|^{-\alpha}d\xi<+\infty$, then $U^\alpha \mu\in C_\infty(\mathbb{R}^n)$.
			\item[(ii)] If $\int_{\mathbb{R}^n}|\hat{\mu}(\xi)|d\xi<+\infty$, then $\mu$ is absolutely continuous with respect to $\mathcal{L}^n$ with density $\frac{d\mu}{d\mathcal{L}^n}\in C_\infty(\mathbb{R}^n)$. 
		\end{enumerate}
	\end{lemma}
	\begin{proof} Under the hypothesis of (i), the formula for the Fourier transform of $k_\alpha$ gives
		\begin{equation}\label{E:Fourierpot}
			U^\alpha\mu(x)=c_{n,\alpha}\int_{\mathbb{R}^n}e^{-i\left\langle x,\xi\right\rangle}\hat{\mu}(\xi)|\xi|^{-\alpha}d\xi,\qquad x\in\mathbb{R}^n.
		\end{equation}
		Both (i) and (ii) now follow from the Riemann--Lebesgue lemma \cite[Proposition 2.2.17]{Grafakos1}.\end{proof}
	
	\begin{corollary}\label{C:Fourier}
		Let $X:[0,1]^k\to\mathbb{R}^n$ be a Borel function.
		\begin{enumerate}
			\item[(i)] If $0< \alpha<n$ and $\int_{\mathbb{R}^n}|\hat{\mu}_{X}(\xi)||\xi|^{-\alpha}d\xi<+\infty$, then $U^\alpha \mu_{X}\in C_\infty(\mathbb{R}^n)$.
			\item[(ii)] If $\int_{\mathbb{R}^n}|\hat{\mu}_{X}(\xi)|d\xi<+\infty$, then $X$ has continuous local times $L_{X}\in C_\infty(\mathbb{R}^n)$. 
		\end{enumerate}
	\end{corollary}
	
	\begin{remark}
		Corollary \ref{C:Fourier} (ii) is Berman's original result \cite[Lemma 2.2]{Berman69}. The fact that $L_{X}\in C_\infty(\mathbb{R}^n)$ implies $U^\alpha \mu_{X}\in C_\infty(\mathbb{R}^n)$ can be seen from a straightforward calculation; for $0<\alpha\leq 2$ it is related to the well-known Feller property of the rotationally symmetric $\alpha$-stable semigroup, whose potential operator is a constant multiple of $U^\alpha$, see \cite[Section 4.1]{Jacob}. The hypothesis in (i) is weaker than that in (ii). 
	\end{remark}
	
	For fractional Brownian fields Corollary \ref{C:Fourier} gives the following partial generalization of \cite[Lemma 3.3]{Berman69}.

	\begin{lemma}\label{L:Bermanbasic}
		Let $B^H$ be a fractional Brownian $(k,n)$-field with Hurst index $0<H<1$ over a probability space $(\Omega,\mathcal{F},\mathbb{P})$.
		\begin{enumerate}
			\item[(i)] If $n-\frac{k}{2H}<\alpha<n$, then there is an event $\Omega_0\in\mathcal{F}$ with $\mathbb{P}(\Omega_0)=1$ such that for any $\omega\in\Omega_0$ we have $U^\alpha\mu_{B^H(\cdot,\omega)}\in C_\infty(\mathbb{R}^n)$ for the occupation measure $\mu_{B^H(\cdot,\omega)}$ of $B^H(\cdot,\omega)$. In particular, $U^\alpha\mu_{B^H(\cdot,\omega)}$ is bounded for such $\omega$.
			\item[(ii)] If $n<\frac{k}{2H}$, then there is an event $\Omega_0\in\mathcal{F}$ with $\mathbb{P}(\Omega_0)=1$ such that for any $\omega\in\Omega_0$ the function $B^H(\cdot,\omega)$ has local times $L_{B^H(\cdot,\omega)}\in C_\infty(\mathbb{R}^n)$. In particular, $L_{B^H(\cdot,\omega)}$ is bounded for such $\omega$.
		\end{enumerate}
	\end{lemma}
	
	For $k=n=1$ Lemma \ref{L:Bermanbasic} (ii) was stated and proved in  \cite[Lemma 3.3]{Berman69}; similar arguments can be found in \cite[Chapter 14]{Kahane}. For convenience we sketch the arguments for item (i).
	
	\begin{proof}[Proof of Lemma \ref{L:Bermanbasic} (i)]
		Since by Cauchy--Schwarz we have  
		\begin{multline}
			\mathbb{E}\int_{\mathbb{R}^n\setminus B(0,1)} |\hat{\mu}_{B^H}(\xi)||\xi|^{-\alpha}d\xi\notag\\
			\leq \left(\mathbb{E}\int_{\mathbb{R}^n\setminus B(0,1)}|\hat{\mu}_{B^H}(\xi)|^2|\xi|^{n+\varepsilon-2\alpha}d\xi\right)^{\frac12}\left(\int_{\mathbb{R}^n\setminus B(0,1)}|\xi|^{-n-\varepsilon}d\xi\right)^{\frac12},
		\end{multline}
		it suffices to observe that we can find $\varepsilon>0$ such that 
		\begin{equation}\label{E:eps}
			H<\frac{k}{2(n-\alpha)+\varepsilon}
		\end{equation}
		and, as a consequence, 
		\begin{align}\label{E:constC}
			\mathbb{E}\int_{\mathbb{R}^n}|\hat{\mu}_{B^H}(\xi)|^2 &|\xi|^{n+\varepsilon-2\alpha}d\xi 
			\notag\\
			&=\int_{\mathbb{R}^n}\int_{[0,1]^k}\int_{[0,1]^k}\mathbb{E}\exp\left(i\left\langle \xi,B^H(t)-B^H(s)\right\rangle\right)\:dsdt\:|\xi|^{n+\varepsilon-2\alpha}d\xi\notag\\
			&=\int_{\mathbb{R}^n}\int_{[0,1]^k}\int_{[0,1]^k}\exp\left(-\frac12|t-s|^{2H}|\xi|^2\right)\:dsdt\:|\xi|^{n+\varepsilon-2\alpha}d\xi\notag\\
			&=\int_{[0,1]^k}\int_{[0,1]^k}\frac{dsdt}{|t-s|^{(2n+\varepsilon-2\alpha)H}}\int_{\mathbb{R}^n}\exp\left(-\frac12|\eta|^2\right)|\eta|^{n+\varepsilon-2\alpha}d\eta\notag\\
			&=:C(n,H,\alpha,\varepsilon)
		\end{align}
		is finite. 
	\end{proof}
	
	Lemma \ref{L:Bermanbasic} applies only to a restricted range of indices $H$. However, from its proof explicit bounds can be extracted.
	
	Given $0<\delta<1$ and $1\leq \ell<+\infty$, let
	\begin{equation}\label{E:Gagliardo}
		[u]_{\delta,\ell}:=\left(\int_{(0,1)^k}\int_{(0,1)^k}\frac{|u(t)-u(s)|^\ell}{|t-s|^{k+\delta\ell}}dsdt\right)^{\frac{1}{\ell}},\qquad u\in L^\ell((0,1)^k,\mathbb{R}^n),
	\end{equation}
	and 
	\begin{equation}\label{E:Sobonorm}
		\|u\|_{W^{\delta,\ell}}:=\left(\int_{(0,1)^k}|u(t)|^\ell dt+[u]_{\delta,\ell}^\ell\right)^{\frac{1}{\ell}},\qquad u\in L^\ell((0,1)^k,\mathbb{R}^n).
	\end{equation}
	The \emph{fractional Sobolev space $W^{\delta,\ell}((0,1)^k,\mathbb{R}^n)$} is the space of all $u\in L^\ell((0,1)^k,\mathbb{R}^n)$ such that $\|u\|_{W^{\delta,\ell}}<+\infty$. In this section we abbreviate notation and write $W^{\delta,\ell}$. If $\delta>\frac{k}{\ell}$, then each $u\in W^{\delta,\ell}$ has a H\"older continuous representative of order $\delta-\frac{k}{\ell}$ and, with the notation from \eqref{E:Hoelder},
	\begin{equation}\label{E:Soboembedd}
		\|u\|_{\mathcal{C}_0^{\delta-\frac{k}{\ell}}}\leq c(k,\delta,\ell)\|u\|_{W^{\delta,\ell}},\qquad u\in W^{\delta,\ell},
	\end{equation}
	where $c(k,\delta,\ell)$ is a constant depending only on $k$, $\delta$ and $\ell$.
	
	We write $\omega_{n-1}$ for the surface area of the unit sphere in $\mathbb{R}^n$.
	
	\begin{lemma}\label{L:Bermanexplicit}
		Let $B^H$ be a fractional Brownian $(k,n)$-field with Hurst index $0<H<1$ over a probability space $(\Omega,\mathcal{F},\mathbb{P})$. Let $n-\frac{k}{2H}<\alpha<n$, let $\varepsilon>0$ be as in \eqref{E:eps} and let $C(n,H,\alpha,\varepsilon)$ be as in \eqref{E:constC}.
		\begin{enumerate}
			\item[(i)]  For any
			$M>M_0(n,H,\alpha,\varepsilon):=c_{n,\alpha}\Big(\frac{\omega_{n-1}}{n-\alpha}+\sqrt{\frac{\omega_{n-1}\,C(n,H,\alpha,\varepsilon)}{\varepsilon}}\Big)$
			there is an event $A_0(M)\in\mathcal{F}$ with $\mathbb{P}(A_0(M))>0$ such that for all $\omega\in A_0(M)$ we have
			\begin{equation}\label{E:subBH}
				\sup_{x\in\mathbb{R}^n}U^\alpha\mu_{B^H(\cdot,\omega)}(x)<M.
			\end{equation}
			\item[(ii)] If $\ell>\frac{2k}{H}$ is an even integer, then for any 
			\begin{equation}\label{E:M1}
				M>M_0(n,H,\alpha,\varepsilon)+n^{\frac{\ell}{2}}(\ell-1)!!(k^{\frac{\ell H}{2}}+1)
			\end{equation}
			there is an event $A_1(M)\in\mathcal{F}$ with $\mathbb{P}(A_1(M))>0$ such that \eqref{E:subBH} holds for all $\omega\in A_1(M)$ and, in addition, $B^H(\cdot,\omega)\in \mathcal{C}_0^{H-2k/\ell}([0,1]^k,\mathbb{R}^n)$ with
			\begin{equation}\label{E:rho1}
				\|B^H(\cdot,\omega)\|_{\mathcal{C}_0^{H-2k/\ell}}\leq c\Big(k,H-\frac{k}{\ell},\ell\Big)M^{1/\ell},
			\end{equation}
			where $c(k,H-\frac{k}{\ell},\ell)$ is as in \eqref{E:Soboembedd}.
		\end{enumerate}
	\end{lemma}
	
	\begin{proof}
		For (i), let $A_0(M):=\{c_{n,\alpha}\int_{\mathbb{R}^n}|\hat{\mu}_{B^H}(\xi)||\xi|^{-\alpha}d\xi<  M\}$. Markov's inequality, the estimates in the preceding proof, together with $\int_{B(0,1)}|\xi|^{-\alpha}d\xi=\frac{\omega_{n-1}}{n-\alpha}$ and $\int_{\mathbb{R}^n\setminus B(0,1)}|\xi|^{-n-\varepsilon}d\xi=\frac{\omega_{n-1}}{\varepsilon}$, give
		\[\mathbb{P}(A_0(M)^c)\leq \frac{c_{n,\alpha}}{M}\mathbb{E}\int_{\mathbb{R}^n} |\hat{\mu}_{B^H}(\xi)||\xi|^{-\alpha}d\xi\leq  \frac{M_0(n,H,\alpha,\varepsilon)}{M}<1,\]
		and for any $\omega\in A_0(M)$ and $x\in\mathbb{R}^n$ we have 
		\[U^\alpha\mu_{B^H(\cdot,\omega)}(x)=c_{n,\alpha}\left|\int_{\mathbb{R}^n} e^{-i\left\langle x,\xi\right\rangle}\hat{\mu}_{B^H}(\xi)|\xi|^{-\alpha}d\xi\right|\leq c_{n,\alpha}\int_{\mathbb{R}^n}|\hat{\mu}_{B^H}(\xi)||\xi|^{-\alpha}d\xi<M;\]
		here we have used \eqref{E:Fourierpot}.
		For (ii), let $A_1(M):=\{c_{n,\alpha}\int_{\mathbb{R}^n}|\hat{\mu}_{B^H}(\xi)||\xi|^{-\alpha}d\xi+\|B^H\|_{W^{H-\frac{k}{\ell},\ell}}^\ell<M\}$. We can proceed as before, taking into account that 
		\begin{align}
			\mathbb{E}\|B^H\|^\ell_{W^{H-\frac{k}{\ell},\ell}}&=\int_{[0,1]^k}\mathbb{E}|B^H(t)|^\ell dt+\int_{[0,1]^k}\int_{[0,1]^k}\frac{\mathbb{E}|B^H(t)-B^H(s)|^\ell}{|t-s|^{\ell H}}dsdt\notag\\
			&\leq n^{\frac{\ell}{2}}(\ell-1)!!k^{\frac{\ell H}{2}}+ n^{\frac{\ell}{2}}(\ell-1)!!\notag
		\end{align}
		by \eqref{E:genGauss} and \eqref{E:indepcomp}.
	\end{proof}
	
	The next result is more conveniently formulated and proved in terms of Bessel potentials. Given $\alpha>0$, consider the \emph{Bessel kernel $g_\alpha$ of order $\alpha$}, determined by its Fourier transform 
	$\hat{g}_\alpha(\xi)=\left\langle \xi\right\rangle^{-\alpha}$, where $\left\langle \xi \right\rangle=(1+|\xi|^2)^{1/2}$, $\xi\in\mathbb{R}^n$.
	For any $R>0$ we can find a constant $c_R>0$ such that 
	\begin{equation}\label{E:kernelcomp}
		k_\alpha(x)\leq c_R g_\alpha(x),\quad x\in B(0,R);
	\end{equation}
	this is due to the well-known comparison of Riesz and Bessel kernels, \cite[Section 1.2.4]{AH96}.
	Given a Radon measure $\mu$ on $\mathbb{R}^n$, we write 
	$G^\alpha\mu(x):=\int_{\mathbb{R}^n} g_\alpha(x-y)\mu(dy)$, $x\in\mathbb{R}^n$,
	for its \emph{Bessel potential of order $\alpha$}.

	 Suppose that $X:[0,1]^k\times \Omega\to\mathbb{R}^n$ is a centered Gaussian random field over a probability space $(\Omega,\mathcal{F},\mathbb{P})$ such that the covariance matrices $\sigma^2(s,t):=\mathrm{Cov}(X(t)-X(s))$ and $\sigma^2(t):=\mathrm{Cov}(X(t))$ have positive determinants at $\mathcal{L}^{2k}$-a.e. $(s,t)\in [0,1]^{2k}$ and $\mathcal{L}^{k}$-a.e. $t\in [0,1]^{k}$, respectively.  Recall that $X$ is said to be \emph{locally nondeterministic} if for any $\ell\geq 2$ there are constants $c_\ell>0$ and $\varepsilon_\ell>0$ such that 
	\begin{equation}\label{E:lnd}
		\mathrm{Var}\Big(\sum_{j=2}^\ell \left\langle u_j,\sigma_j^{-1}(X(t_j)-X(t_{j-1}))\right\rangle + \left\langle u_1,\sigma_1^{-1}X(t_1)\right\rangle\Big)\geq c_\ell\:\sum_{j=1}^\ell |u_j|^2
	\end{equation}
	for all $u_1,\dots,u_\ell\in \mathbb{R}^n$ and distinct $t_1,\dots,t_\ell\in [0,1]^k$ with all $t_j$ lying in a single subcube of $[0,1]^k$ of side length $\varepsilon_\ell$ and such that
	\begin{equation}
		\label{E:bad-condition}
		|t_{j+1}-t_j|\leq |t_{j+1}-t_i|, \quad 1\leq i\leq j<\ell.
	\end{equation}
	Here 
	\begin{equation}\label{E:sigmaj}
		\sigma_j^2:=\mathrm{Cov}(X(t_j)-X(t_{j-1})),\quad j=2,\dots,\ell\quad \text{and}\quad \sigma_1^2:=\mathrm{Cov}(X(t_1)),
	\end{equation}
	$\sigma_j$ is the symmetric positive definite square root of $\sigma_j^2$, in particular $\sigma_j\sigma_j^T=\sigma_j^2$, and $\sigma_j^{-1}$ is the inverse of $\sigma_j$. See \cite{Pitt78} or \cite[Section 24]{GemanHorowitz}. 
	
	The following is an analogue of \cite[Theorem 4]{Pitt78} for potentials of occupation measures. 
	
	\begin{theorem}\label{T:Pitt}
		Let $X:[0,1]^k\times \Omega\to\mathbb{R}^n$ be a centered Gaussian random field over a probability space $(\Omega,\mathcal{F},\mathbb{P})$. Suppose that  $\sigma^2(s,t)$ and $\sigma^2(t)$ have positive determinants at $\mathcal{L}^{2k}$-a.e. $(s,t)\in [0,1]^{2k}$ and $\mathcal{L}^{k}$-a.e. $t\in [0,1]^{k}$, respectively, and that $X$ is locally nondeterministic. If $0<\alpha<n$ and there is some $\delta>0$ such that
		\begin{equation}\label{E:covcond}
			\int_{[0,1]^k}\frac{dt}{|\det \sigma^2(t)|^{\frac{n-\alpha}{2n}+\delta}}+\esssup_{s\in [0,1]^k}\int_{[0,1]^k}\frac{dt}{|\det \sigma^2(s,t)|^{\frac{n-\alpha}{2n}+\delta}}<+\infty,
		\end{equation}
		then $G^\alpha\mu_{X}$ is locally H\"older continuous $\mathbb{P}$-a.s. 
	\end{theorem}	
	\begin{remark}
		Theorem \ref{T:Pitt} should be compared to Pitt's prominent classical result \cite[Theorem 4]{Pitt78}, also discussed in \cite[Proposition 25.12]{GemanHorowitz}. This result states that if 
		\begin{equation}\label{E:Pitt}
			\int_{[0,1]^k}\frac{dt}{|\det \sigma^2(t)|^{\frac{1}{2}+\delta}}+\esssup_{s\in [0,1]^k}\int_{[0,1]^k}\frac{dt}{|\det \sigma^2(s,t)|^{\frac{1}{2}+\delta}}<+\infty
		\end{equation}
		with some $\delta>0$, then $X$ has local times $L_X$ which are locally H\"older continuous $\mathbb{P}$-a.s. Since \eqref{E:covcond} is weaker than \eqref{E:Pitt}, Theorem \ref{T:Pitt} can still give results even if \eqref{E:Pitt} fails to hold.
	\end{remark}

	Theorem \ref{T:Pitt} is proved by a variant of Pitt's argument \cite{GemanHorowitz, Pitt78}.
	
	\begin{proof}[Proof of Theorem \ref{T:Pitt}]
		Let $\ell\geq 2$ be an even integer. To simplify notation, we assume that $\varepsilon_\ell=1$ in \eqref{E:lnd}. Note that for general $\varepsilon_\ell>0$ we could cover $[0,1]^k$ by finitely many cubes $Q$ of side length at most $\varepsilon_\ell$ with pairwise disjoint interiors, carry out the proof for the measures $\mu_X^Q(B):=\mathcal{L}^k(\{t\in Q: X(t)\in B\})$, $B\subset \mathbb{R}^n$ Borel, and use the fact that $\mu_X=\sum_Q\mu_X^Q$.

		We use the notations $\bar{t}=(t_1,\dots,t_\ell)\in (\mathbb{R}^k)^\ell$ and $\bar{x}=(x_1,\dots,x_{\ell})\in (\mathbb{R}^n)^\ell$. For distinct $t_1,\dots,t_\ell\in [0,1]^k$ the $(\mathbb{R}^n)^\ell$-valued random vector $(X(t_1),\dots,X(t_\ell))$ has a density $p_\ell(\bar{t};\cdot)$, that is,
		\[\mathbb{P}(X(t_1)\in B_1,\dots,X(t_\ell)\in B_\ell)=\int_{B_1\times\cdots  \times B_\ell}p_\ell(\bar{t};\bar{x})d\bar{x},\quad B_j\subset \mathbb{R}^n\ \text{Borel},\ j=1,\dots,\ell.\]
The set of all $\bar{t}=(t_1,\dots,t_\ell)\in ([0,1]^k)^\ell$ such that $t_i=t_j$ for some $i\neq j$ is an $\mathcal{L}^{k\ell}$-null set, and we define $p_\ell(\bar{t};\cdot)$ to be zero on this null set.

The computation below, applied to
		\begin{equation*}
			\mathbb{E}[(G^\alpha\mu_X(x))^\ell]=\int_{([0,1]^k)^\ell}\int_{(\mathbb{R}^n)^\ell}\prod_{j=1}^\ell g_\alpha(x_j-x)\,p_\ell(\bar{t};\bar{x})\,d\bar{x}\,d\bar{t}
		\end{equation*}
		(that is, with all terms involving $y$ omitted and with $\theta=0$), shows that $\mathbb{E}[(G^\alpha\mu_X(x))^\ell]<+\infty$ for all $x\in\mathbb{R}^n$. In particular, $G^\alpha\mu_X(x)$ is finite a.s., and the following uses of Fubini's theorem are justified. For any $x,y\in \mathbb{R}^n$ we have
		\begin{align}
			\mathbb{E}[G^\alpha\mu_{X}(x)-G^\alpha\mu_{X}(y)]^\ell
			&=\mathbb{E}\left[\int_{\mathbb{R}^n} g_\alpha(z-x)\mu_{X}(dz)-\int_{\mathbb{R}^n}g_\alpha(z-y)\mu_{X}(dz)\right]^\ell\notag\\
			&=\mathbb{E}\left[\prod_{j=1}^\ell \int_{[0,1]^k}\big\lbrace g_\alpha(X(t_j)-x)- g_\alpha(X(t_j)-y)\big\rbrace dt_j\right]\notag\\
			&=\int_{([0,1]^k)^\ell}\int_{(\mathbb{R}^n)^\ell}\prod_{j=1}^\ell \big\lbrace g_\alpha(x_j-x)- g_\alpha(x_j-y)\big\rbrace p_\ell(\bar{t};\bar{x})\:d\bar{x}\:d\bar{t}.\notag
		\end{align}
		 For $\bar{t}=(t_1,\dots,t_\ell)\in ([0,1]^k)^\ell$ with distinct components the inner integral in the last line equals 
		\begin{multline}
			\int_{(\mathbb{R}^n)^{\ell-1}}\prod_{j=1}^{\ell-1} \big\lbrace g_\alpha(x_j-x)- g_\alpha(x_j-y)\big\rbrace\times\notag\\
			\qquad\times\Big(G^{\alpha,\cdot_\ell}p_\ell(\bar{t};x_1,\dots,x_{\ell-1},\cdot_\ell)(x)-G^{\alpha,\cdot_\ell}p_\ell(\bar{t};x_1,\dots,x_{\ell-1},\cdot_\ell)(y)\Big)\:d(x_1,\dots,x_{\ell-1}).
		\end{multline}
		This can be expanded further, first into
		\begin{multline}
			\int_{(\mathbb{R}^n)^{\ell-2}}\prod_{j=1}^{\ell-2} \big\lbrace g_\alpha(x_j-x)- g_\alpha(x_j-y)\big\rbrace\Big(G^{\alpha,\cdot_{\ell-1}}G^{\alpha,\cdot_\ell}p_\ell(\bar{t};x_1,\dots,x_{\ell-2},\cdot_{\ell-1},\cdot_\ell)(x,x)\notag\\
			-G^{\alpha,\cdot_{\ell-1}}G^{\alpha,\cdot_\ell}p_\ell(\bar{t};x_1,\dots,x_{\ell-2},\cdot_{\ell-1},\cdot_\ell)(x,y)\notag\\
			-G^{\alpha,\cdot_{\ell-1}}G^{\alpha,\cdot_\ell}p_\ell(\bar{t};x_1,\dots,x_{\ell-2},\cdot_{\ell-1},\cdot_\ell)(y,x)\notag\\
			+G^{\alpha,\cdot_{\ell-1}}G^{\alpha,\cdot_\ell}p_\ell(\bar{t};x_1,\dots,x_{\ell-2},\cdot_{\ell-1},\cdot_\ell)(y,y)\Big)d(x_1,\dots,x_{\ell-2}),\notag
		\end{multline}
		and finally into 
		\[\sum_{j=0}^\ell(-1)^j \sum_{r=1}^{\binom{\ell}{j}}G^{\alpha,\cdot_1}\cdots G^{\alpha,\cdot_\ell}p_\ell(\bar{t};\cdot_1,\dots,\cdot_\ell)(\bar{z}_{j,r}),\]
		where $G^{\alpha,\cdot_j}$ means that the Bessel potential operator is applied with respect to $x_j$ and where, as in \cite[p.~43]{GemanHorowitz}, $\bar{z}_{j,r}$ runs through all points having $x$ in exactly $j$ coordinates and $y$ in the $\ell-j$ remaining ones. Following \cite{GemanHorowitz, Pitt78}, we write
		\[\hat{p}_{\ell}(\bar{t};\bar{\xi})=\exp\Big(-\frac12 \mathrm{Var}\sum_{j=1}^\ell \big\langle \xi_j,X(t_j)\big\rangle\Big)\]
		for the characteristic function of $X(\bar{t})$, that is, for the Fourier transform with respect to $\bar{x}=(x_1,\dots,x_\ell)\in (\mathbb{R}^n)^\ell$. Then
		\[(G^{\alpha,\cdot_1}\cdots G^{\alpha,\cdot_\ell}p_\ell(\bar{t};\cdot_1,\dots,\cdot_\ell))^\wedge(\bar{\xi})=\prod_{m=1}^\ell \left\langle\xi_m\right\rangle^{-\alpha}\hat{p}_\ell(\bar{t};\bar{\xi}),\]
		again with the Fourier transform taken on $(\mathbb{R}^n)^\ell$. Fourier inversion gives 
		\[G^{\alpha,\cdot_1}\cdots G^{\alpha,\cdot_\ell}p_\ell(\bar{t};\bar{z}_{j,r})=(2\pi)^{-\ell n}\int_{(\mathbb{R}^n)^\ell}\prod_{m=1}^\ell \left\langle\xi_m\right\rangle^{-\alpha}\hat{p}_\ell(\bar{t};\bar{\xi})e^{-i\left\langle \bar{\xi},\bar{z}_{j,r}\right\rangle}d\bar{\xi}.\]
		Plugging in and then substituting $y=x+w$,
		\begin{align}
			\mathbb{E}&[G^\alpha\mu_{X}(x)-G^\alpha\mu_{X}(y)]^\ell\notag\\
			&=(2\pi)^{-\ell n}\int_{([0,1]^k)^\ell}\int_{(\mathbb{R}^n)^\ell}\prod_{m=1}^\ell \left\langle\xi_m\right\rangle^{-\alpha}\hat{p}_\ell(\bar{t};\bar{\xi})\sum_{j=0}^\ell(-1)^j \sum_{r=1}^{\binom{\ell}{j}}e^{-i\left\langle \bar{\xi},\bar{z}_{j,r}\right\rangle}d\bar{\xi}\:d\bar{t}\notag\\
			&=(2\pi)^{-\ell n}\int_{([0,1]^k)^\ell}\int_{(\mathbb{R}^n)^\ell}\prod_{m=1}^\ell \left\langle\xi_m\right\rangle^{-\alpha}\hat{p}_\ell(\bar{t};\bar{\xi})e^{-i\left\langle \bar{\xi},\underline{x}\right\rangle}\sum_{j=0}^\ell(-1)^j \sum_{r=1}^{\binom{\ell}{j}}e^{-i\left\langle\bar{\xi},\bar{\zeta}_{j,r}\right\rangle}d\bar{\xi}\:d\bar{t},\notag
		\end{align}
		where $\underline{x}=(x,\dots,x)$ and $\bar{\zeta}_{j,r}$ has zero and $w$ at the positions where $\bar{z}_{j,r}$ has $x$ and $y$. An inductive argument \cite[p.~44]{GemanHorowitz} shows that 
		\[\sum_{j=0}^\ell(-1)^j \sum_{r=1}^{\binom{\ell}{j}}e^{-i\left\langle\bar{\xi},\bar{\zeta}_{j,r}\right\rangle}=\prod_{j=1}^\ell(1-e^{-i\left\langle\xi_j,w\right\rangle}).\]
		Using the fact that for any $0<\theta\leq 1$ and $t\in \mathbb{R}$ we have $|e^{it}-1|\leq 2|t|^\theta$, we find that 
		$|1-e^{-i\left\langle\xi_j,w\right\rangle}|\leq 2|\left\langle\xi_j,w\right\rangle|^\theta\leq 2|\xi_j|^\theta|w|^\theta$,
		and since $|e^{-i\left\langle \bar{\xi},\underline{x}\right\rangle}|=1$ and $|\xi_j|^\theta\leq \left\langle\xi_j\right\rangle^\theta$, this gives
		\begin{multline}\label{E:intermed}
			\mathbb{E}[G^\alpha\mu_{X}(x)-G^\alpha\mu_{X}(y)]^\ell\\
			\leq 2^\ell(2\pi)^{-\ell n}|w|^{\theta\ell}\:\int_{([0,1]^k)^\ell}\int_{(\mathbb{R}^n)^\ell}\prod_{m=1}^\ell \left\langle\xi_m\right\rangle^{\theta-\alpha}\exp\Big(-\frac12 \mathrm{Var}\sum_{j=1}^\ell \big\langle \xi_j,X(t_j)\big\rangle\Big)\:d\bar{\xi}\:d\bar{t}.
		\end{multline}
		For fixed $\bar{t}=(t_1,t_2,\ldots,t_{\ell})$ with distinct components we now estimate the inner integral. One can always perform a permutation of the indices to guarantee \eqref{E:bad-condition} holds: choose $t_\ell$ arbitrarily and, recursively, let $t_{j-1}$ be a point among the remaining ones closest to $t_j$, see also \cite[p.~41]{GemanHorowitz}. Notice that both domains of integration of the inner and outer integrals on the right-hand side of \eqref{E:intermed} are invariant under permutations, and so is the product   $\prod_{m=1}^\ell\left\langle\xi_m\right\rangle^{\theta-\alpha}$. Therefore we may assume that $\bar{t}$ satisfies \eqref{E:bad-condition}. The use of these permutations leads to a factor $\ell!$ in \eqref{E:finalestPitt} below.
		
		The quantity in (\ref{E:covcond}) remains finite if in place of $\delta$ we use some $0<\delta'<\delta$, because $a^{-p'}\leq a^{-p}+1$ for $a>0$ and $0<p'<p$; so we may assume that $0< 2\delta<\frac{\alpha}{n}$. Now set 
		$q:=\frac{n}{\alpha-2\delta n}$
		and choose $0<\theta< 2\delta n$. Obviously $q>1$, and by H\"older's inequality the inner integral in the right-hand side of (\ref{E:intermed}) is bounded by
		\begin{equation}\label{E:product}
			\left(\int_{(\mathbb{R}^n)^\ell}\prod_{m=1}^\ell \left\langle\xi_m\right\rangle^{q(\theta-\alpha)}\:d\bar{\xi}\right)^{\frac{1}{q}}\left(\int_{(\mathbb{R}^n)^\ell}\exp\Big(-\frac{q'}{2} \mathrm{Var}\sum_{j=1}^\ell \big\langle \xi_j,X(t_j)\big\rangle\Big)\:d\bar{\xi}\right)^{\frac{1}{q'}},
		\end{equation}
		where $\frac{1}{q}+\frac{1}{q'}=1$. Since $q>\frac{n}{\alpha-\theta}$, the first integral in (\ref{E:product}) is finite. For the second one, we follow \cite[Section 6]{Pitt78}. Changing variables such that  $\xi_j=\eta_j-\eta_{j+1}$, $j=1,\dots,\ell-1$, and  $\xi_\ell=\eta_\ell$, we obtain
		\[\sum_{j=1}^\ell \big\langle \xi_j,X(t_j)\big\rangle=  \big\langle \eta_1,X(t_1)\big\rangle+\sum_{j=2}^\ell \big\langle \eta_j,X(t_j)-X(t_{j-1})\big\rangle,\]
		and, by \eqref{E:lnd}, 
		\[\mathrm{Var}\Big(\sum_{j=1}^\ell\big\langle \xi_j,X(t_j)\big\rangle\Big)\geq c_\ell\:\sum_{j=1}^\ell |\sigma_j\eta_j|^2,\]
		where $c_\ell>0$ is a constant and $\sigma_j$ is the root of \eqref{E:sigmaj}. Consequently the second integral in 
		(\ref{E:product}) is bounded by 
		\[\int_{(\mathbb{R}^n)^\ell} \exp\Big(-\frac{c_\ell\:q'}{2}\sum_{j=1}^\ell |\sigma_j\eta_j|^2\Big)d\bar{\eta}=\prod_{j=1}^\ell\frac{1}{|\det(\sigma_j)|}\int_{(\mathbb{R}^n)^\ell} \exp\Big(-\frac{c_\ell\:q'}{2}\sum_{j=1}^\ell |\zeta_j|^2\Big)d\bar{\zeta},\]
		here we have used the change of variables $\zeta_j=\sigma_j\eta_j$. 
		
		Since the integral on the right-hand side equals 
		$\Big(\int_{\mathbb{R}^n}\exp\Big(-\frac{c_\ell\:q'}{2}|\zeta|^2\Big)d\zeta\Big)^\ell$,
		the product (\ref{E:product}) is bounded by 
		$C\:\prod_{j=1}^\ell|\det(\sigma_j)|^{-1/q'}$
		with a constant $C>0$ depending only on $n$, $\alpha$, $\ell$, $\theta$, $\delta$ and the constant $c_\ell$ from \eqref{E:lnd}. 
		
		Since $\frac{1}{q'}=\frac{n-\alpha}{n}+2\delta$, an evaluation of (\ref{E:intermed}) gives
		\begin{align}\label{E:finalestPitt}
				\mathbb{E}[G^\alpha\mu_{X}(x)&-G^\alpha\mu_{X}(y)]^\ell\notag\\
		 &\leq C\:\ell!\,|x-y|^{\theta\ell}\:\int_{([0,1]^k)^\ell}\frac{1}{|\det(\sigma(t_1))|^{\frac{n-\alpha}{n}+2\delta}}\prod_{j=2}^\ell\frac{1}{|\det(\sigma(t_{j-1},t_j))|^{\frac{n-\alpha}{n}+2\delta}}d\bar{t}\notag\\
			&\leq C\:\ell!\,|x-y|^{\theta\ell} \int_{[0,1]^k}\frac{dt}{|\det \sigma^2(t)|^{\frac{n-\alpha}{2n}+\delta}}\Big(\esssup_{s\in [0,1]^k}\int_{[0,1]^k}\frac{dt}{|\det \sigma^2(s,t)|^{\frac{n-\alpha}{2n}+\delta}}\Big)^{\ell-1}.
		\end{align}
		
		Since $\delta$ and $\theta$ do not depend on $\ell$, we may choose the even integer $\ell>\frac{n}{\theta}$. Then, given $0<\gamma'<\theta-\frac{n}{\ell}$, the Kolmogorov--Chentsov theorem ensures that $(x,\omega)\mapsto G^\alpha\mu_{X(\cdot,\omega)}(x)$ has a modification $(x,\omega)\mapsto \tilde{G}(x,\omega)$ such that for any cube in $\mathbb{R}^n$ there is a random variable $K>0$ and an event $\Omega_0$ of full probability such that
		$|\tilde{G}(x,\omega)-\tilde{G}(y,\omega)|\leq K(\omega)\:|x-y|^{\gamma'}$
		for all $x,y$ from that cube and all $\omega\in \Omega_0$. See \cite{Kunita90} or \cite{Pitt78}. Covering $\mathbb{R}^n$ by countably many cubes gives an event $\Omega_1$ of full probability on which $\tilde{G}$ is continuous on $\mathbb{R}^n$.

		 Actually $\tilde{G}$ and $G^\alpha\mu_X$ are indistinguishable: By Fubini's theorem there is an event $\Omega_2\subset\Omega_1$ of full probability such that for all $\omega\in \Omega_2$ we have $\tilde{G}(\cdot,\omega)=G^\alpha\mu_{X(\cdot,\omega)}$ outside an $\mathcal{L}^n$-null set $N(\omega)\subset \mathbb{R}^n$. Since $G^\alpha\mu_{X(\cdot,\omega)}$ is lower semicontinuous, see \cite[Chapter I, Theorem 1.3]{Landkof}, it follows that $G^\alpha\mu_{X(\cdot,\omega)}(x)\leq\tilde{G}(x,\omega)<+\infty$ at all $x\in \mathbb{R}^n$. The fact that the radial kernel $g_\alpha$ is integrable and continuous off the origin implies that
\[\lim_{r\to 0}\frac{1}{\mathcal{L}^n(B(x,r))}\int_{B(x,r)}g_\alpha(z-y)dy=g_\alpha(z-x),\quad z\in\mathbb{R}^n\setminus \{x\}.\]		 
There is a constant $c>0$ depending only on $n$ and $\alpha$ such that 
\[\frac{1}{\mathcal{L}^n(B(x,r))}\int_{B(x,r)}g_\alpha(z-y)dy\leq c\:g_{\alpha}(z-x),\qquad r\leq 1,\]
see \cite[p.~157]{AH96}. By the preceding $g_{\alpha}(\cdot-x)$ is $\mu_{X(\cdot,\omega)}$-integrable. Therefore Fubini's theorem and dominated convergence give
\begin{align}
\tilde{G}(x,\omega)&=\lim_{r\to 0}\frac{1}{\mathcal{L}^n(B(x,r))}\int_{B(x,r)}\tilde{G}(y,\omega)\:dy\notag\\
&=\lim_{r\to 0}\frac{1}{\mathcal{L}^n(B(x,r))}\int_{B(x,r)}G^\alpha\mu_{X(\cdot,\omega)}(y)\:dy\notag\\
&=\lim_{r\to 0}\int_{\mathbb{R}^n}\frac{1}{\mathcal{L}^n(B(x,r))}\int_{B(x,r)}g_\alpha(z-y)\:dy\:\mu_{X(\cdot,\omega)}(dz)\notag\\
&=\int_{\mathbb{R}^n}g_\alpha(z-x)\mu_{X(\cdot,\omega)}(dz)\notag\\
&=G^\alpha\mu_{X(\cdot,\omega)}(x)\notag
\end{align}		 
for all $\omega\in \Omega_2$ and all $x\in\mathbb{R}^n$.	
	\end{proof}
	
	\begin{remark}
		In applications the value of the constant $c_\ell>0$ in \eqref{E:lnd} is typically unknown; one can only show its existence, see \cite[Proposition 7.2]{Pitt78}. As a consequence, also the value of the constant $C$ in \eqref{E:finalestPitt} remains unknown.
	\end{remark}

	Again fractional Brownian fields provide a class of examples. We consider the increments $X(t)=B^H(t_0+t)-B^H(t_0)$ with a base point $t_0\in(0,1]^k$. Since the increments of $B^H$ are stationary, $X$ has the same law as the restriction of $B^H$ to $[0,1]^k$, and the shift ensures that the covariance matrices $\sigma^2(t)$ of $B^H(t_0+\cdot)$ are nondegenerate for all $t\in[0,1]^k$. Item (i) in the following is due to Theorem \ref{T:Pitt}, while item (ii) follows from \cite[Proposition 7.3]{Pitt78}.

	\begin{lemma}\label{L:BermanPitt} 
		Let $B^H$ be a fractional Brownian $(k,n)$-field with Hurst index $0<H<1$ over a probability space $(\Omega,\mathcal{F},\mathbb{P})$, let $t_0\in (0,1]^k$ and set 
$X(t):=B^H(t_0+t)-B^H(t_0)$, $t\in [0,1]^k$.
		\begin{enumerate}
			\item[(i)] If $n-\frac{k}{H}<\alpha< n$ and $0<\gamma<H$, then there is some $\Omega_1\in\mathcal{F}$ with $\mathbb{P}(\Omega_1)=1$ such that for any $\omega\in\Omega_1$ we have $X(\cdot,\omega)\in \mathcal{C}_0^\gamma([0,1]^k,\mathbb{R}^n)$ and $G^\alpha\mu_{X(\cdot,\omega)}\in C(\mathbb{R}^n)$. For any such $\omega$ the potential $G^\alpha\mu_{X(\cdot,\omega)}$ is bounded on $\mathbb{R}^n$, and the same is true for $U^\alpha\mu_{X(\cdot,\omega)}$.
			\item[(ii)] If $n<\frac{k}{H}$ and $0<\gamma<H$, then there is an event $\Omega_1\in\mathcal{F}$ with $\mathbb{P}(\Omega_1)=1$ such that for any $\omega\in\Omega_1$ we have $X(\cdot,\omega)\in \mathcal{C}_0^\gamma([0,1]^k,\mathbb{R}^n)$ and the function $X(\cdot,\omega)$ has local times $L_{X(\cdot,\omega)}\in C(\mathbb{R}^n)$. In particular, $L_{X(\cdot,\omega)}$ is locally bounded for any such $\omega$.
		\end{enumerate}
	\end{lemma}

	\begin{proof}[Proof of Lemma \ref{L:BermanPitt} (i)] 
		A variant of \cite[Proposition 7.2]{Pitt78} shows that $B^H(t_0+\cdot)$ is locally nondeterministic. The matrices $\sigma^2(s,t)$ and $\sigma^2(t)$ are diagonal with all $n$ diagonal elements equal to $|t-s|^{2H}$ respectively $|t_0+t|^{2H}$ (for $B^H(t_0+\cdot)$), and since $H(n-\alpha)<k$ by hypothesis, one can find some $\delta>0$ such that \eqref{E:covcond} holds. By the Kolmogorov--Chentsov theorem and by Theorem \ref{T:Pitt}, applied to $B^H(t_0+\cdot)$, we can find an event $\Omega_1$ of full probability such that for any $\omega\in \Omega_1$ we have $X(\cdot,\omega)\in \mathcal{C}_0^\gamma([0,1]^k,\mathbb{R}^n)$ and $G^\alpha\mu_{X(\cdot,\omega)}=G^\alpha\mu_{B^H(t_0+\cdot,\omega)}(B^H(t_0,\omega)+\cdot)\in C(\mathbb{R}^n)$.
Choosing $R_\omega>0$ large enough to have $(X(\cdot,\omega))([0,1]^k)\subset B(0,R_\omega)$ and $M_\omega>0$ large enough to have $G^\alpha\mu_{X(\cdot,\omega)}(x)< M_\omega$ for all $x\in B(0,2R_\omega)$, we find that
		\begin{equation}\label{E:BesselR}
			\int_{|x-y|<R_\omega} g_\alpha(x-y)\mu_{X(\cdot,\omega)}(dy)<M_\omega,\qquad x\in\mathbb{R}^n,
		\end{equation} 
		because for $x\in \mathbb{R}^n\setminus B(0,2R_\omega)$ this integral is zero. In addition,  we have
		\[\int_{|x-y|\geq R_\omega}g_\alpha(x-y)\mu_{X(\cdot,\omega)}(dy)\leq \sup_{|z|\geq R_\omega}g_\alpha(z)<+\infty,\qquad x\in\mathbb{R}^n,\]
		since $\mu_{X(\cdot,\omega)}$ is a probability measure and $g_\alpha$ is bounded outside a neighborhood of the origin. Consequently $G^\alpha\mu_{X(\cdot,\omega)}$ is bounded. Combining \eqref{E:kernelcomp}, \eqref{E:BesselR} and the bound $k_\alpha(z)\leq R_\omega^{\alpha-n}$, $|z|\geq R_\omega$, shows that also $U^\alpha\mu_{X(\cdot,\omega)}$ is bounded.
	\end{proof}

	\section{Complements on nonlinear compositions}\label{S:comp}
	
	In this auxiliary section we extend a nonlinear composition result from \cite{HTV2020, HTV2021-1} to a higher-dimensional situation; the result is stated in Theorem \ref{T:main} below. Throughout this section, $x$ and $y$ denote points of a domain $G\subset\mathbb{R}^k$, while $\xi$, $\eta$, $\zeta$ and $z$ denote points of $\mathbb{R}^n$.
	
	Recall that an element $\varphi$ of $L^1(\mathbb{R}^n)$ is said to be of bounded variation, $\varphi\in BV(\mathbb{R}^n)$, if its distributional partial derivatives $D_i\varphi$ are signed Radon measures and its $\mathbb{R}^n$-valued gradient measure $D\varphi=(D_1\varphi,\dots,D_n\varphi)$ has a finite total variation measure $\|D\varphi\|$. 
	
	Let $G\subset \mathbb{R}^k$ be a bounded domain. Given $0<\theta<1$ and $1\leq q<+\infty$, let the Gagliardo seminorm $u\mapsto [u]_{\theta,q}$ and the norm $u\mapsto \|u\|_{W^{\theta,q}}$ be defined as in \eqref{E:Gagliardo} and \eqref{E:Sobonorm}, but with $\theta$, $q$ and the more general $G$ in place of $\delta$, $\ell$ and $(0,1)^k$. By $W^{\theta,q}(G,\mathbb{R}^n)$ we denote the space of all $u\in L^q(G,\mathbb{R}^n)$ such that $\|u\|_{W^{\theta,q}}$ is finite. We complement this by writing
	$[u]_{\theta,\infty}:=\sup_{x,y\in G,\ x\neq y} \frac{|u(y) - u(x)|}{|y-x|^\theta}$
	and defining $W^{\theta,\infty}(G,\mathbb{R}^n)$ to be the space of all Borel functions $u:G\to \mathbb{R}^n$ for which $\|u\|_{W^{\theta,\infty}}:=\sup_{x\in G}|u(x)|+[u]_{\theta,\infty}$ is finite, that is, the space of all bounded functions on $G$ which are H\"older continuous of order $\theta$. If $n=1$ we omit $\mathbb{R}^n$ and simply write $W^{\theta,q}(G)$ and $W^{\theta,\infty}(G)$.
	
	\begin{remark}
		In the notation of \eqref{E:Hoelder}, $\mathcal{C}^\theta_0([0,1]^k,\mathbb{R}^n)$ consists of the unique continuous extensions to $[0,1]^k$ of those $u\in W^{\theta,\infty}((0,1)^k,\mathbb{R}^n)$ whose extension satisfies $u(0)=0$, and $\|u\|_{\mathcal{C}^\theta_0}=[u]_{\theta,\infty}$. Here we prefer the notation $u\mapsto [u]_{\theta,\infty}$, because it is only a seminorm on $W^{\theta,\infty}(G,\mathbb{R}^n)$.
	\end{remark}
	
	Let $L^0(G,\mathbb{R}^n)$ denote the space of $\mathcal{L}^k$-equivalence classes of $\mathbb{R}^n$-valued measurable functions on $G$. For any $u\in L^0(G,\mathbb{R}^n)$ we set $\mu_u(B):=\mathcal{L}^k(\{x\in G: u(x)\in B\})$, $B\subset \mathbb{R}^n$ Borel; this generalizes \eqref{E:occumeas}. As before, we call $\mu_u$ the \emph{occupation measure of $u$}. The measure $\mu_u$ is still finite, but in general not a probability measure.
	
	Given $\varphi\in BV(\mathbb{R}^n)$ and a fractional Sobolev function $u\in W^{\theta,q}(G,\mathbb{R}^n)$, we are interested in a well-defined composition $\varphi\circ u$ and in a basic regularity estimate for this composition. Since a priori $\varphi$ is defined only as an $\mathcal{L}^{n}$-equivalence class, not even the definition is immediately clear, let alone a regularity result. However, both will be available if an in general nonlinear interaction energy of $\|D\varphi\|$ and the occupation measure $\mu_u$ of $u$ is finite. In \cite{HTV2020, HTV2021-1} we discussed the special case $k=1$; here we extend these results to the case $k\geq 1$, which might be of interest in the context of partial differential equations.

	Let $1\leq p<+\infty$, $0<s<1$ and $\varphi\in BV(\mathbb{R}^n)$. We consider the functional 
	\begin{equation}\label{E:V}
		V_{\varphi,s,p}(u):=\int_{G}(U^{1-s}\|D\varphi\|(u(x)))^pdx,\qquad u\in L^0(G,\mathbb{R}^n),
	\end{equation}
	and complement \eqref{E:V} by defining $V_{\varphi,s,\infty}(u)$ to be the $\mathcal{L}^k$-essential supremum of $U^{1-s}\|D\varphi\|(u(\cdot))$ and we write $V(\varphi,s,p):=\{u\in L^0(G,\mathbb{R}^n):\ V_{\varphi,s,p}(u)<+\infty\}$. For functions $u$ defined everywhere, such as continuous $u$, we define $\mathbb{V}_{\varphi,s,\infty}(u)$ to be the supremum of $U^{1-s}\|D\varphi\|(u(\cdot))$ and $\mathbb{V}(\varphi,s,\infty)$ as the set of all Borel functions on $G$ such that this supremum is finite.

	\begin{remark}\mbox{}
		\begin{enumerate}
			\item[(i)] In comparison to the functional in Remark \ref{R:nonlin}, definition \eqref{E:V} amounts to using $X=u$ and 
			$\nu=\|D\varphi\|$ and switching the roles of $\mu_X$ and $\nu$.
			\item[(ii)] The quantity $V_{\varphi,s,p}(u)$ in \eqref{E:V} is a nonlinear energy, see \cite[p.~36]{AH96}. In the linear case
			$p=1$ we recover the mutual interaction energy defined in \eqref{E:repattenergy},
			\begin{equation}\label{E:justsame}
				V_{\varphi,s,1}(u)=I^{1-s}(\mu_u,\|D\varphi\|).
			\end{equation}
			\item[(iii)] It is easily seen that $V(\varphi,s,p_1)\subset V(\varphi,s,p_2)$ for $p_1>p_2$ and $V(\varphi,s_1,p)\subset V(\varphi,s_2,p)$ for $s_1>s_2$.
		\end{enumerate}
	\end{remark}
	
	Recall that $\varphi\in L^1_{\loc}(\mathbb{R}^n)$ is said to have an \emph{approximate limit} at $\xi\in \mathbb{R}^n$ if there exists $\lambda_\varphi(\xi)\in\mathbb{R}$ such that
	$\lim_{r\to 0}\frac{1}{\mathcal{L}^n(B(\xi,r))}\int_{B(\xi,r)}|\varphi(z)-\lambda_{\varphi}(\xi)|\, dz=0$.
	In this situation, the unique value $\lambda_\varphi(\xi)$ is called the \emph{approximate limit} of $\varphi$ at $\xi$.
	The set of points $\xi\in \mathbb{R}^n$ for which this property does not hold is called the \emph{approximate discontinuity set} (or \emph{exceptional set}) and is denoted by $S_\varphi$. This set $S_\varphi$ does not depend on the choice of the representative for $\varphi$. If $\widetilde{\varphi}$ is a representative of $\varphi\in L^1_{\loc}(\mathbb{R}^n)$, then a point $\xi\not\in S_\varphi$ with $\widetilde{\varphi}(\xi)=\lambda_\varphi(\xi)$ is called a \emph{Lebesgue point} of $\widetilde{\varphi}$, and the set of all Lebesgue points of $\widetilde{\varphi}$ is called the \emph{Lebesgue set} of $\widetilde{\varphi}$. See for instance \cite[Definition 3.63]{AFP}. The set $S_\varphi$ is Borel and of zero Lebesgue measure, see \cite[Proposition 3.64]{AFP}. We say that a Borel function $\widetilde{\varphi}:\mathbb{R}^n\to\mathbb{R}$ is a \emph{Lebesgue representative} of $\varphi\in L^1_{\loc}(\mathbb{R}^n)$ if $\widetilde{\varphi}(\xi)=\lambda_{\varphi}(\xi)$, $\xi\in \mathbb{R}^n\setminus S_\varphi$. The following lemma shows that $S_\varphi$ is contained in the set where $U^{1-s}\|D\varphi\|$ is infinite.

	\begin{lemma}\label{L:finitepot}
		Let $0<s<1$ and $\varphi\in BV(\mathbb{R}^n)$. If $\xi\in\mathbb{R}^n$ satisfies $U^{1-s}\|D\varphi\|(\xi)<+\infty$, then $\xi\in\mathbb{R}^n\setminus S_\varphi$.
	\end{lemma}

	\begin{proof}
		For $r>0$ write $\varphi_r:=\frac{1}{\mathcal{L}^n(B(\xi,r))}\int_{B(\xi,r)}\varphi(z)\,dz$. By the Poincar\'e inequality for $BV$-functions on balls, see for instance \cite[Section 5.6]{EG}, and since $|\xi-\eta|^{1-s-n}\geq r^{1-s-n}$ for $\eta\in B(\xi,r)$,
		\[\frac{1}{\mathcal{L}^n(B(\xi,r))}\int_{B(\xi,r)}|\varphi(z)-\varphi_r|\,dz\le c(n)\,r^{1-n}\|D\varphi\|(B(\xi,r))\le c(n)\,r^s\,U^{1-s}\|D\varphi\|(\xi).\]
		Consequently $|\varphi_\rho-\varphi_r|\le 2^nc(n)\,r^s\,U^{1-s}\|D\varphi\|(\xi)$ for $\frac{r}{2}\le\rho\le r$, which implies that
		$\lambda:=\lim_{r\to 0}\varphi_r$ exists. The preceding estimate then gives
		\begin{equation*}
			\lim_{r\to 0}\frac{1}{\mathcal{L}^n(B(\xi,r))}\int_{B(\xi,r)}|\varphi(z)-\lambda|\,dz=0.
		\end{equation*}
	\end{proof}

	As a consequence, we obtain the following; for $k=1$ see also \cite[Proposition 4.4]{HTV2020}.
	
	\begin{lemma}\label{lem:key}
		Let $0<s<1$, $\varphi \in BV(\mathbb{R}^n)$ and $u\in V(\varphi,s,1)$. Then $S_\varphi$ is a null set for the occupation measure $\mu_u$ of $u$,
		$\mu_u(S_\varphi)=0$.
	\end{lemma}

	\begin{proof}
		Since $u\in V(\varphi,s,1)$, we have $U^{1-s}\|D\varphi\|(u(x))<+\infty$ for $\mathcal{L}^k$-a.e. $x\in G$. By Lemma \ref{L:finitepot} this implies $u(x)\in\mathbb{R}^n\setminus S_\varphi$ for $\mathcal{L}^k$-a.e. $x\in G$, that is, $\mu_u(S_\varphi)=0$.
	\end{proof}
	
	If $\varphi$ and $u$ are as in Lemma \ref{lem:key}, then for any two Lebesgue representatives $\widetilde{\varphi}_1$ and $\widetilde{\varphi}_2$ the 
	compositions $\widetilde{\varphi}_1\circ u$ and $\widetilde{\varphi}_2 \circ u$ define the same $\mathcal{L}^k$-equivalence class, see \cite[Lemma 2.4]{HTV2020}. We can therefore define the composition $\varphi\circ u\in L^0(G)$ as the $\mathcal{L}^k$-equivalence class of $\widetilde{\varphi}\circ u$, where $\widetilde{\varphi}$ is an arbitrary Lebesgue representative of $\varphi$.
	
	The following estimate for nonlinear compositions is an extension of \cite[Proposition 4.6]{CLV16}, \cite[Proposition 4.27]{HTV2020} and
	\cite[Proposition 5.18]{HTV2021-1}. 
	
	\begin{theorem}\label{T:main} Let $G\subset \mathbb{R}^k$ be a bounded domain.
		Let $0<s<1$ and $1\leq p,q,r\leq +\infty$ be such that
		\begin{equation}\label{eq:parametertrading}
			\frac{1}{p}+\frac{s}{q}\le\frac{1}{r}
		\end{equation}
		with the convention that $\frac{1}{\infty}:=0$. If $0<\theta<1$, $\varphi\in BV(\mathbb{R}^n)$ and $u\in W^{\theta,q}(G,\mathbb{R}^n)\cap V(\varphi,s,p)$, then for any $0<\beta<\theta s$, we have for $p<+\infty$ 
		\begin{equation}\label{eq:mainestimate}
			[\varphi\circ u]_{\beta,r}\leq c\:[u]_{\theta,q}^{s}(V_{\varphi,s,p}(u))^{\frac{1}{p}},
		\end{equation}
		and
		for $p=+\infty$,
		\begin{equation}\label{eq:mainestimate2}
			[\varphi\circ u]_{\beta,r}\leq c\:[u]_{\theta,q}^{s}V_{\varphi,s,\infty}(u),
		\end{equation}
		where $c>0$ is a constant depending only on $k$, $n$, $s$, $p$, $q$, $r$, $\beta$, $\theta$ and $G$, and where for $p=q=+\infty$ we replace 		
		$V(\varphi,s,\infty)$ and $V_{\varphi,s,\infty}(u)$ by $\mathbb{V}(\varphi,s,\infty)$ and $\mathbb{V}_{\varphi,s,\infty}(u)$. The results remain true for $\beta=\theta s$ if $r=+\infty$ or if $r<+\infty$ and $q=sr$.
	\end{theorem}
	
	\begin{remark} Estimate \eqref{eq:mainestimate} differs from usual composition estimates \cite{Bourdaud2019,Runst1996,Brezis2001-1,Brezis2001-2,Mazya2002,Bourdaud2010,Sickel1996,Bourdaud2011}. It does not provide the boundedness of the composition operator acting on a single Sobolev space, but balances some lack of regularity of $\varphi$ by a sufficient irregularity of $u$ at sites where $\varphi$ is \enquote{bad}.
	\end{remark}
	
	\begin{corollary}
		\label{cor:sobolev-membership} Let $G\subset \mathbb{R}^k$ be a bounded domain.
		Let $0<s,\theta <1$, $1\leq p,q,r \leq+\infty$ be such that \eqref{eq:parametertrading} holds, let $\varphi\in BV(\mathbb{R}^n)$ and $u \in W^{\theta,q}(G,\mathbb{R}^n)\cap V(\varphi,s,p)$.  For $p=q=+\infty$ we replace $V(\varphi,s,\infty)$ by $\mathbb{V}(\varphi,s,\infty)$. Then for any $0<\beta<s\theta$, the composition $\varphi \circ u$ is an element of $W^{\beta,r}(G)$. This remains true for $\beta=\theta s$ if $r=+\infty$ or if $r<+\infty$ and $q=sr$.
	\end{corollary}
	
	\begin{proof}
		For $r<+\infty$ it suffices to note that, under the stated hypotheses, $\varphi \circ u \in L^{r}(G)$. This was shown for $k=1$ in \cite[Lemma 4.30]{HTV2020} and \cite[Proposition 5.19]{HTV2021-1}; the proofs there work also for $k\geq 1$ and for $p=+\infty$. For $r=+\infty$ this is clear since $G$ is bounded.
	\end{proof}

	For any given Borel measure $\nu$ on $\mathbb{R}^n$, any $0\leq a\leq n$, and any $0<R\leq +\infty$, let $\mathcal{M}_{a,R} \nu(\xi):=\sup_{0<r<R}r^{a-n}\:\nu(B(\xi,r))$, $\xi\in \mathbb{R}^n$,
denote the \emph{(truncated) fractional Hardy--Littlewood maximal function} of $\nu$ of order $a$. Our main tools are the estimate
	\begin{equation}\label{E:backend}
		|\widetilde{\varphi}(\xi)-\widetilde{\varphi}(\eta)|
		\leq c(n,s)|\xi-\eta|^s\left[\mathcal{M}_{1-s,4|\xi-\eta|}\left\|D\varphi\right\|(\xi)+\mathcal{M}_{1-s,4|\xi-\eta|}\left\|D\varphi\right\|(\eta)\right],
	\end{equation}
	valid for all $0<s<1$, any Lebesgue representative $\widetilde{\varphi}$ of $\varphi$ and all $\xi,\eta\in \mathbb{R}^n\setminus S_\varphi$; the left-hand side does not depend on the choice of $\widetilde{\varphi}$, since $\widetilde{\varphi}=\lambda_\varphi$ on $\mathbb{R}^n\setminus S_\varphi$. We also use the trivial estimate
	\begin{equation}\label{eq:maximal-function-bound-new}
		\mathcal{M}_{1-s, 4|\xi-\eta|}\left\|D\varphi\right\|(\xi)\leq \int_{\mathbb{R}^n}\frac{\left\|D\varphi\right\|(d\zeta)}{|\xi-\zeta|^{n-1+s}} = U^{1-s}\Vert D\varphi\Vert(\xi),\qquad \xi,\eta\in \mathbb{R}^n.
	\end{equation}
	A proof of \eqref{E:backend} can be found in \cite[Proposition C.1]{HTV2020}; it is based on \cite[Lemma 4.1 and Corollary 4.3]{AK}. The constant $c(n,s)>0$ in \eqref{E:backend} depends only on $n$ and $s$. 
	
	\begin{proof}[Proof of Theorem \ref{T:main}] 
		If $r=+\infty$, then $p=q=+\infty$ by \eqref{eq:parametertrading}.  Let $u\in W^{\theta,\infty}(G,\mathbb{R}^n)\cap \mathbb{V}(\varphi,s,\infty)$. By definition, $u$ is $\theta$-H\"older continuous. As $G$ is bounded, we have $u\in V(\varphi,s,1)$, so that $\varphi\circ u$ is defined, and by Lemma \ref{L:finitepot} we have $u(x)\in\mathbb{R}^n\setminus S_\varphi$ for all $x\in G$. Let $\widetilde{\varphi}$ be a Lebesgue representative of $\varphi$. By \eqref{E:backend}, for $x,y\in G$,
			\begin{align*}
				&		|\widetilde{\varphi}(u(x))-\widetilde{\varphi}(u(y))|\\
				&\leq c(n,s)|u(x)-u(y)|^s\left[\mathcal{M}_{1-s,4|u(x)-u(y)|}\left\|D\varphi\right\|(u(x))+\mathcal{M}_{1-s,4|u(x)-u(y)|}\left\|D\varphi\right\|(u(y))\right]\\
				&\leq c(n,s)[u]_{\theta,\infty}^s |x-y|^{\theta s} \left[U^{1-s}\|D\varphi\|(u(x))+U^{1-s}\|D\varphi\|(u(y))\right]\\
				&\leq 2 c(n,s)[u]_{\theta,\infty}^s |x-y|^{\theta s} \sup_{z\in G}U^{1-s} \|D\varphi\|(u(z))\\
				&\leq 2 c(n,s)[u]_{\theta,\infty}^s |x-y|^{\theta s}  \mathbb{V}_{\varphi,s,\infty}(u).		\end{align*}
			Consequently $\widetilde{\varphi}\circ u$ is a representative of $\varphi\circ u$, and for any $0<\beta\leq\theta s$ we obtain $[\varphi\circ u]_{\beta,\infty}\leq 2c(n,s)\,(\diam G)^{\theta s-\beta}\,[u]_{\theta,\infty}^s\,\mathbb{V}_{\varphi,s,\infty}(u)$, which is \eqref{eq:mainestimate2}.

Let us now assume that $r<+\infty$. By Lemma \ref{lem:key} we have $u(x)\in\mathbb{R}^n\setminus S_\varphi$ for $\mathcal{L}^k$-a.e. $x\in G$. Let $\widetilde{\varphi}$ be a Lebesgue representative of $\varphi$; then $\widetilde{\varphi}\circ u$ is a representative of $\varphi\circ u$. Using \eqref{E:backend}, \eqref{eq:maximal-function-bound-new}, the inequality $(a+b)^r\leq 2^{r-1}(a^r+b^r)$, $a,b\geq 0$, and the symmetry in $x$ and $y$, we obtain
		\begin{align}
						\int_{G}&\int_{G} \frac{|\widetilde{\varphi}(u(x))-\widetilde{\varphi}(u(y))|^r}{|x-y|^{k+\beta r}}\,dy\,dx\notag\\
			&\leq c(n,s)^r\int_{G}\int_{G}\frac{|u(x)-u(y)|^{sr}}{|x-y|^{k+\beta r}}\notag\\
			&\qquad\times\left[\mathcal{M}_{1-s,4|u(x)-u(y)|}\left\|D\varphi\right\|(u(x))+\mathcal{M}_{1-s,4|u(x)-u(y)|}\left\|D\varphi\right\|(u(y))\right]^r\,dy\,dx\notag\\
			&\leq 2^rc(n,s)^r\int_{G} (U^{1-s}\left\|D\varphi\right\|(u(x)))^r \int_{G}\frac{|u(x)-u(y)|^{sr}}{|x-y|^{k+\beta r}}\,dy\,dx.\label{E:common}
		\end{align}
		If $q=+\infty$, then $|u(x)-u(y)|\leq [u]_{\theta,\infty}|x-y|^\theta$, and the right-hand side of \eqref{E:common} is bounded by
		\begin{equation}\label{E:cutout-new}
			2^rc(n,s)^r[u]_{\theta,\infty}^{rs}\int_{G} (U^{1-s}\left\|D\varphi\right\|(u(x)))^r \int_{G}|x-y|^{-k+(\theta s-\beta) r}\,dy\,dx.
		\end{equation}
		For $p=+\infty$ this is bounded by $c\:[u]_{\theta,\infty}^{rs} (V_{\varphi,s,\infty}(u))^r$, while for $p<+\infty$ H\"older's inequality gives the bound $c\:[u]_{\theta,\infty}^{rs} (V_{\varphi,s,p}(u))^{r/p}$; note that in this case $r\leq p$ by \eqref{eq:parametertrading}.
		
		Now suppose that $q<+\infty$. Let $\ell:=\frac{q}{sr}$;
		we have $\ell\ge 1$ by \eqref{eq:parametertrading}. If $\ell=1$, then $p=+\infty$, and since $\beta r\leq \theta s r=\theta q$ and $|x-y|\leq \diam G$ for $x,y\in G$, the right-hand side of \eqref{E:common} is bounded by
		\begin{equation*}
			2^rc(n,s)^r\,(1\vee\diam G)^{\theta q-\beta r}\,(V_{\varphi,s,\infty}(u))^r\int_G\int_{G}\frac{|u(x)-u(y)|^{q}}{|x-y|^{k+\theta q}}\,dy\,dx.
		\end{equation*}
		It remains to consider the case $\ell>1$. For any $0<\delta\leq k$ and $x\in G$, the measure
		$\nu^x_\delta(dy):=\mathbf{1}_{G}(y)|x-y|^{\delta-k}\,dy$
		is finite, and $C_\delta:=\sup_{x\in G}\nu^x_\delta(G)<+\infty$. An application of Jensen's inequality with respect to the probability measure $\frac{\nu^x_\delta}{\nu^x_\delta(G)}$ and the convex function $t\mapsto t^\ell$,  followed by an application of H\"older's inequality, gives
		\begin{align}\label{E:block}
			&\int_G (U^{1-s}\left\|D\varphi\right\|(u(x)))^r \int_{G}\frac{|u(x)-u(y)|^{sr}}{|x-y|^{k+\beta r}}\,dy\,dx\notag\\
			&= \int_G (U^{1-s}\left\|D\varphi\right\|(u(x)))^r \left(\int_{G}\frac{|u(x)-u(y)|^{sr}\nu_\delta^x(G)}{|x-y|^{\beta r+\delta}}\,\frac{\nu_\delta^x(dy)}{\nu_\delta^x(G)}\right)\,dx\notag\\
			&\leq  \int_G (\nu^x_\delta(G))^{\frac{\ell-1}{\ell}} (U^{1-s}\left\|D\varphi\right\|(u(x)))^r \left(\int_{G}\frac{|u(x)-u(y)|^{sr\ell}}{|x-y|^{k+\beta r \ell+\delta \ell-\delta}}\,dy\right)^{\frac{1}{\ell}}\,dx\notag\\
			&\leq  \left(\int_G \nu^x_\delta(G) (U^{1-s}\left\|D\varphi\right\|(u(x)))^{\frac{r\ell}{\ell-1}}\,dx\right)^{\frac{\ell-1}{\ell}}\left(\int_G\int_{G}\frac{|u(x)-u(y)|^{sr\ell}}{|x-y|^{k+\beta r\ell+\delta \ell-\delta}}\,dy\,dx\right)^{\frac{1}{\ell}}\notag\\
			&\leq  C_\delta^{\frac{\ell-1}{\ell}}\left(\int_G (U^{1-s}\left\|D\varphi\right\|(u(x)))^{\frac{r\ell}{\ell-1}}\,dx\right)^{\frac{\ell-1}{\ell}}\left(\int_G\int_{G}\frac{|u(x)-u(y)|^{sr\ell}}{|x-y|^{k+\beta r\ell+\delta \ell-\delta}}\,dy\,dx\right)^{\frac{1}{\ell}}.
		\end{align}
		Choose $0<\delta\leq k$ such that
		\begin{equation}\label{E:deltachoice}
			\theta-\frac{\delta(\ell-1)}{sr\ell}\ge\frac{\beta}{s}.
		\end{equation}
		Then
		\begin{equation}\label{E:comparenumbers}
			k+\theta q=k+\theta sr\ell\ge k+\left(\frac{\beta}{s}+\frac{\delta}{sr}-\frac{\delta}{sr\ell}\right)sr\ell=k+\beta r\ell+\delta \ell-\delta.
		\end{equation}
		Since $\ell=\frac{q}{sr}$, condition \eqref{eq:parametertrading} is equivalent to $p\ge\frac{r\ell}{\ell-1}$, and H\"older's inequality gives $\|\cdot\|_{L^{\frac{r\ell}{\ell-1}}(G)}\le (\mathcal{L}^k(G))^{\frac{\ell-1}{r\ell}-\frac{1}{p}}\|\cdot\|_{L^p(G)}$. Using this and \eqref{E:comparenumbers}, we see that the last line in \eqref{E:block} is bounded by
		\begin{equation*}
			c\:(V_{\varphi,s,p}(u))^{r/p}\left(\int_G\int_{G}\frac{|u(x)-u(y)|^{q}}{|x-y|^{k+\theta q}}\,dy\,dx\right)^{\frac{rs}{q}}.
		\end{equation*}
	\end{proof}
	
	\begin{remark}\label{R:safety}
		The right-hand side in \eqref{E:cutout-new} blows up as $\beta$ approaches $s\theta$. A choice of a positive $\delta$ in \eqref{E:deltachoice} is possible only if $\beta<\theta s$. In the cases that $r=+\infty$ or that $r<+\infty$ and $q=sr$, for which $\beta=\theta s$ is admissible, condition \eqref{eq:parametertrading} requires $p=+\infty$.
	\end{remark}

		Using \eqref{E:backend} and \eqref{eq:maximal-function-bound-new}, the steps of the proof of Theorem \ref{T:main} can be used to produce the following continuity result in a rather straightforward way.
		
		\begin{corollary}\label{C:continuity}
			Let $G\subset \mathbb{R}^k$ be a bounded domain.
			Let $0<s<1$ and $1\leq p,q,r\leq +\infty$ be such that \eqref{eq:parametertrading} holds. Then for any $\varphi\in BV(\mathbb{R}^n)$ and any $u,v\in L^q(G,\mathbb{R}^n)\cap V(\varphi,s,p)$, we have for $p<+\infty$
			\[
			\|\varphi\circ u-\varphi\circ v\|_{L^r(G)}\le c\|u-v\|_{L^q(G,\mathbb{R}^n)}^s\left(V_{\varphi,s,p}(u)^{\frac{1}{p}}+V_{\varphi,s,p}(v)^{\frac{1}{p}}\right),
			\]
			and for $p=+\infty$,
			\[
			\|\varphi\circ u-\varphi\circ v\|_{L^r(G)}\le c\|u-v\|_{L^q(G,\mathbb{R}^n)}^s\left(V_{\varphi,s,\infty}(u)+V_{\varphi,s,\infty}(v)\right),
			\]
			where $c>0$ is a constant depending only on $k,n,s,p,q$, $r$ and $G$.  If $u,v\in\mathbb{V}(\varphi,s,\infty)$, then $u(G)\cup v(G)\subset\mathbb{R}^n\setminus S_\varphi$, the compositions $\widetilde{\varphi}\circ u$ and $\widetilde{\varphi}\circ v$ do not depend on the choice of the Lebesgue representative $\widetilde{\varphi}$ of $\varphi$, and
			\begin{equation}\label{E:supest}
			\sup_{x\in G}|\widetilde{\varphi}(u(x))-\widetilde{\varphi}(v(x))|\le c(n,s)\sup_{x\in G}|u(x)-v(x)|^s\left(\mathbb{V}_{\varphi,s,\infty}(u)+\mathbb{V}_{\varphi,s,\infty}(v)\right),
			\end{equation}
			where $c(n,s)$ is as in \eqref{E:backend}.

		\end{corollary}

		\begin{proof}
			The first two estimates follow from \eqref{E:backend}, \eqref{eq:maximal-function-bound-new}, Lemma \ref{lem:key} and H\"older's inequality. If $u,v\in\mathbb{V}(\varphi,s,\infty)$, then $u(x),v(x)\in\mathbb{R}^n\setminus S_\varphi$ for all $x\in G$ by Lemma \ref{L:finitepot}. Consequently $\widetilde{\varphi}(u(x))=\lambda_\varphi(u(x))$ and $\widetilde{\varphi}(v(x))=\lambda_\varphi(v(x))$ for any Lebesgue representative $\widetilde{\varphi}$ of $\varphi$, and \eqref{E:backend} and \eqref{eq:maximal-function-bound-new}, applied with $\xi=u(x)$ and $\eta=v(x)$, give \eqref{E:supest}.
		\end{proof}

	\section*{Declarations}
	
	\noindent 	\textbf{Ethical Approval:}   Not applicable.
	
	\noindent\textbf{Competing interests:} The authors declare no competing interests.
	
	\noindent 	\textbf{Authors' contributions:} All authors have contributed equally.
	
	\noindent 	\textbf{Funding:}  The research of MH was supported in part by the DFG IRTG 2235 \enquote{Searching for the regular in the irregular: Analysis of singular and random systems} and by the DFG CRC 1283, \enquote{Taming uncertainty and profiting from randomness and low regularity in analysis, stochastics and their applications}. JMT gratefully acknowledges travel support from the Magnus Ehrnrooth Foundation. The research of JMT and LV was partially supported by the European Union's Horizon Europe research and innovation programme under the Marie Sk\l{}odowska-Curie Actions Staff Exchanges (Grant agreement no.~101183168 -- LiBERA, Call: HORIZON-MSCA-2023-SE-01).
	
	\noindent
	\textbf{Disclaimer:}
	Co-funded by the European Union. Views and opinions expressed are however those of the authors only and do not necessarily reflect those of
	the European Union or the European Education and Culture Executive Agency (EACEA). Neither the European Union nor EACEA can be held responsible for them.
	
	\noindent 	\textbf{Availability of data and materials:} Not applicable.
	
	\noindent \textbf{AI Declaration:} Large language models have been used for checking spelling, grammar and the mathematical arguments.

\end{document}